\documentclass[12pt]{amsart}
\usepackage{graphicx}
\usepackage[centertags]{amsmath}
\usepackage{amsfonts}
\usepackage{amssymb}
\usepackage{amsthm}
\newtheorem{thm}{Theorem}
\newtheorem{lem}{Lemma}[section]
\newtheorem{cor}{Corollary}[section]
\newtheorem{conj}[thm]{Conjecture}
\newtheorem{prop}[lem]{Proposition}
\theoremstyle{definition}
\newtheorem{defn}[lem]{Definition}
\theoremstyle{remark}
\newtheorem{rem}{Remark}[section]
\numberwithin{equation}{section}
\newenvironment{theorem}[2][Theorem]{\begin{trivlist}
\item[\hskip \labelsep {\bfseries #1}\hskip \labelsep {\bfseries #2}]}{\end{trivlist}}

\newcommand{\norm}[1]{\left\Vert#1\right\Vert}

\newcommand{\set}[1]{\left\{#1\right\}}

\newcommand{\pd}[2]{\frac{\partial #1}{\partial #2}}

\newcommand{\into}{\hookrightarrow}
\newcommand{\G}{\Gamma}

\newcommand{\calA}{\mathcal{A}}
\newcommand{\calC}{\mathcal{C}}

\newcommand{\calF}{\mathcal{F}}

\newcommand{\calH}{\mathcal{H}}
\newcommand{\calI}{\mathcal{I}}
\newcommand{\calJ}{\mathcal{J}}

\newcommand{\calM}{\mathcal{M}}
\newcommand{\calO}{\mathcal{O}}
\newcommand{\calP}{\mathcal{P}}
\newcommand{\calR}{\mathcal{R}}

\newcommand{\bbZ}{\mathbb{Z}}

\newcommand{\bbQ}{\mathbb{Q}}
\newcommand{\bbR}{\mathbb R}
\newcommand{\bbC}{\mathbb C}

\newcommand{\bbN}{\mathbb N}

\newcommand{\bbH}{\mathbb{H}}

\newcommand{\Tr}{ \mathrm{tr}}

\newcommand{\SL}{ \mathrm{SL}}
\newcommand{\SO}{ \mathrm{SO}}

\newcommand{\PSL}{ \mathrm{PSL}}
\newcommand{\PGL}{ \mathrm{PGL}}
\newcommand{\SU}{ \mathrm{SU}}
\newcommand{\PSU}{ \mathrm{PSU}}
\newcommand{\PSO}{ \mathrm{PSO}}
\newcommand{\Li}{ \mathrm{Li}}
\newcommand{\Ram}{ \mathrm{Ram}}
\newcommand{\sgn}{ \mathrm{sgn}}
\newcommand{\GL}{ \mbox{GL}}

\newcommand{\Mat}{\mathrm{Mat}}

\newcommand{\vol}{\mathrm{vol}}

\newcommand{\Reg}{\mathrm{Reg}}

\newcommand{\bfE}{\mathbf{E}}
\newcommand{\lap}{\triangle}
\newcommand{\bs}{\backslash}
\newcommand{\id}{1\!\!1}
\newcommand{\reg}{\mathrm{Reg}}

\newcommand{\f}{\mathfrak{f}}

\newcommand{\g}{\mathfrak{g}}
\renewcommand{\Im}{\mathrm{Im}}

\begin{document}
\title[Holonomy of closed geodesics]
{Distribution of holonomy about closed geodesics in a product of hyperbolic planes}%
\author{Dubi Kelmer}%
\address{Department of Mathematics, University of Chicago,  5734 S. University
Avenue Chicago, Illinois 60637}
\email{kelmerdu@math.uchicago.edu}

\thanks{}%

\date{\today}%
\dedicatory{}%
\commby{}
\begin{abstract}
Let $\calM=\Gamma\bs \calH^{(n)}$, where $\calH^{(n)}$ is a product of $n+1$ hyperbolic planes and $\Gamma\subset\PSL(2,\bbR)^{n+1}$ is an irreducible cocompact lattice. We consider closed geodesics on $\calM$ that propagate locally only in one factor. We show that, as the length tends to infinity, the holonomy rotations attached to these geodesics become equidistributed in $\PSO(2)^n$ with respect to a certain measure.
For the special case of lattices derived from quaternion algebras, we can give another interpretation of the holonomy angles under which this measure arises naturally.
\end{abstract}
\maketitle
\section*{Introduction}
In the spirit of the analogy between prime numbers and primitive closed geodesics, Parry and Pollicott proved an analog to the Chebotarev Equidistribution Theorem for the equidistribution of holonomy classes about closed geodesics. By parallel transporting vectors along geodesics, each closed geodesic $C$ on an orientable $d$ dimensional manifold $\calM$ gives rise to a conjugacy class $H_C$ in $SO(d-1)$, called the holonomy class attached to $C$. In \cite{ParryPollicott86}, Parry and Pollicott showed that on a manifold, for which the geodesic flow on the frame bundle is topologically mixing, the holonomy classes become equidistributed in $SO(d-1)$ with respect to Haar measure as the length of the geodesics tends to infinity. In particular, this is true for manifolds with constant negative curvature, that is, the rank one locally symmetric spaces $\calM=\G\bs \SO_0(d,1)/\SO(d)$.

In \cite{SarnakWakayama99}, Sarnak and Wakayama obtained similar results for all other rank one locally symmetric spaces, that is, spaces of the form $\calM=\G\bs G/K$ with $G$ denoting a real semi-simple group of real rank 1, $K\subseteq G$ a maximal compact subgroup, and $\Gamma$ a lattice in $G$. In this setting, the frame flow is not ergodic and the holonomy classes lie in a smaller subgroup $K_0\subset K\subset \SO(d)$. They showed that as the length of the geodesics tends to infinity, the holonomy classes become equidistributed in $K_0$ with respect to Haar measure.

In this paper we study a similar question for the higher rank group $G=\PSL(2,\bbR)^{n+1}$. That is, we address the equidistribution of holonomy classes of closed geodesics on the manifold $\calM=\G\bs G/K=\G\bs \calH^{(n)}$, where
$\calH^{(n)}=\bbH_0\times\bbH_1\times\cdots\times\bbH_n$ is a product of $n+1$ hyperbolic planes, $K=\PSO(2)^{n+1}$, and $\Gamma\subset \PSL(2,\bbR)^{n+1}$ is an irreducible cocompact lattice. The geodesic flow is not ergodic on the unit tangent bundle $T^1\calM$ as each of the functions $E_j(z,\xi)=\langle \xi_j,\xi_j\rangle_{z_j}$ remain constant under the flow (here, $\langle\cdot,\cdot\rangle_{z_j}$ denotes the inner product on the tangent space $T_{z_j}\bbH$).
For each fixed level $\bfE\in [0,\infty)^{n+1}$ the geodesic flow on the corresponding energy shell
$$\Sigma(\bfE)=\set{(z,\xi)\in T\calM|E_j(z,\xi)=E_j}\subset T\calM,$$
is ergodic and we consider the closed geodesics on each one of these shells separately.

In this setting, the holonomy attached to a closed curve on $\calM$ lies in $K=\PSO(2)^{n+1}$. Since this group is commutative, the holonomy does not depend on the base point and each holonomy class is given by $n+1$ rotations.  Moreover, since parallel transporting along a geodesic maps the tangent vector to itself, the holonomy is trivial in any factor where $E_j\neq 0$. Hence, the holonomy of closed geodesics on $\Sigma(\bfE)$ lie in the smaller group $\prod_{E_j=0}\PSO(2)\subset K$.

We will concentrate on the specific shell
$\Sigma_0=\Sigma(1,0,\ldots,0)$, for which the holonomy group $K_0=\PSO(2)^n$ is as big as possible.
For each closed geodesic $C$ on $\Sigma_0$ we attach the holonomy angle $\theta_C\in(-\pi,\pi)^n$ that gives the angle of rotation in each of the nontrivial factors. We are thus interested in the distribution of these angles as the length of the closed geodesics tends to infinity.

Before stating our results we need to introduce some notations (see section \ref{s:trace} for more details).
Let $L^2(\Gamma\bs G)$ denote the space of functions on $G$ satisfying that $f(\gamma
g)=f(g)$ for $\gamma\in \G$ and that $\int_{\Gamma\bs G}|f(g)|^2dg<\infty$.
The algebra of invariant differential operators of $G$ is of rank
$n+1$ and is generated by $n+1$ differential operators denoted by $\Omega_0,\ldots,\Omega_n$, each acting on the corresponding factor. For any $m\in \bbZ^{n+1}$ let $\chi_m$ denote the characters of $K$ given by $\chi_m(k_\theta)=e^{i m\cdot \theta}$.
We denote by $L^2(\G\bs G,m)\subset L^2(\Gamma\bs G)$ the space of functions satisfying $\psi(gk)=\psi(g)\chi_m(k)$.
For any $m=(0,m_1,\ldots,m_n)\in\bbZ^{n+1}$ with all $m_j\neq 0$ we define the subspace $V_m=V_m(\Gamma\bs G)\subseteq L^2(\Gamma\bs G,m)$ by
\begin{equation}\label{e:Vm}
 V_m=\left\{\psi\in L^2(\G\bs G,m)|\forall j\geq 1,\; \Omega_j\psi+|m_j|(1-|m_j|)\psi=0\right\}.
\end{equation}
Let $\{\psi_{k}^{(m)}\}$ be an orthonormal basis for $V_m(\Gamma\bs G)$ composed of eigenfunctions of $\Omega_{0}$ with eigenvalues $$0<\lambda_0^{(m)}\leq \lambda_1^{(m)}\leq\ldots$$
We use the parameter $\alpha_m=\Im(\sqrt{\lambda_0^{(m)}-\tfrac{1}{4}})\in [0,\tfrac{1}{2})$ to measure the  spectral gap for $V_m$ and let $\alpha=\sup_{m}\alpha_m$. The results of Kelmer and Sarnak \cite{KelmerSarnak09} imply that $\alpha_m< 0.34$ for all but finitely many values of $m$; in particular we have that $\alpha<\tfrac{1}{2}$.

\subsection*{Counting closed geodesics}
We can now estimate the number of closed geodesics on $\Sigma_0$ with bounded length.
For any $x>0$ let $\pi(x)$ denote the number of closed geodesics on $\Sigma_0$ with length $l_C\leq x$;
let $\pi_p(x)$ denote the number of such primitive closed geodesics.
\begin{thm}\label{t:PGT0}
As $x\to\infty$,
$$\pi_p(x)=2^{n}Li(e^x)+O(\frac{e^{3 x/4}}{x})+O(\frac{e^{(\alpha_1+1/2)x}}{x}),$$
where $\alpha_1=\max\{\alpha_m|m\in\{\pm1\}^n\}$.
The same result holds for $\pi(x)$ instead of $\pi_p(x)$.
\end{thm}
\begin{rem}
If $\alpha_1\leq \tfrac14$ we only have the first error term. In particular, after Selberg \cite{Selberg65}, this is known when $\Gamma$ is a congruence group. (In fact, in this case the result of Kim and Shahidi \cite{KimShahidi02} imply $\alpha\leq \tfrac{1}{9}$.)
Moreover, the congruence subgroup conjecture implies that all lattices $\Gamma$ in $\PSL(2,\bbR)^{n+1}$ are congruence groups, so we expect the error term to be uniformly bounded by $O(e^{3x/4})$.
Unconditionally, if $\lambda_1^{(1)},\ldots,\lambda_q^{(1)}$ are all eigenvalues in $(0,\frac{3}{16})$, the formula can be corrected to read
\[\pi_p(x)=2^{n}Li(e^{x})+(-1)^{n}\sum_{k=1}^qLi(e^{(c_k+1/2) x})+O(\frac{e^{3x/4}}{x}),\]
with $c_k=\Im(\sqrt{\lambda_k^{(1)}-\tfrac{1}{4}})$.
\end{rem}

\begin{rem}
In \cite{Deitmar09}, Deitmar proved an analogue to the prime geodesic theorem for higher rank spaces for singular geodesics. The main term in the asymptotics of $\pi_p(x)$ can also be obtained by applying his results to this setting.
\end{rem}

\subsection*{Distribution of holonomy about closed geodesics}
After establishing the growth rate of $\pi_p(x)$ we study the distribution of the holonomy angles. We find that the results here are very different from the rank one case described above. In this case, the holonomy angles are not equidistributed in $K_0$ with respect to the Haar measure, but rather behave as if they were representing conjugacy classes in $\PSU(2)^n$. Recall that $[-\pi,\pi]^n$ parameterizes the conjugacy classes in $\PSU(2)^n$ and
let $\mu$ denote the measure on $[-\pi,\pi]^n$ obtained via this parametrization from Haar measure on $\PSU(2)^n$, that is,
\begin{equation}\label{e:mu}
d\mu(\theta)=\prod_{j=1}^n\sin^2(\tfrac{\theta_j}{2})\frac{d\theta_j}{\pi}.
\end{equation}
\begin{thm}\label{t:equi}
For any smooth function $f\in C^\infty((\bbR/2\pi\bbZ)^{n})$
\[\frac{1}{\pi_p(x)}\mathop{\sum_{l_C\leq x}}_{\mathrm{primitive}} f(\theta_{C})=\mu(f)+O(e^{-x/4})+O(e^{(\alpha-1/2)x}).\]
\end{thm}
Since we can approximate the indicator function by smooth functions this implies that
for any $A\subset [-\pi,\pi]^n$ with $\mu(\partial A)=0$ we have
\[\lim_{x\to\infty}\frac{\pi_p(x;A)}{\pi_p(x)}=\mu(A),\]
where $\pi_p(x;A)=\#\{C \;\mathrm{primitive}|l_C\leq x \mbox{ and }\theta_C\in A\}$.
\begin{rem}
The implied constant for the error term in Theorem \ref{t:equi} can be given explicitly in terms of the Fourier coefficients of the test function $f$. In particular, when the set $A\subset [-\pi,\pi]^n$ is sufficiently nice (for example any rectangle) it is possible to give exponential bounds for the error term also for a sharp cutoff (see Corollary \ref{c:equi}).
\end{rem}

\subsection*{Invariance of holonomy angles under sign changes}
The measure $\mu$ defined in (\ref{e:mu}) is invariant under the change of variable $\theta\mapsto \sigma\theta$ for any $\sigma\in\{\pm 1\}^n$, where we denote $\sigma\theta=(\sigma_1\theta_1,\ldots,\sigma_n\theta_n)$. In particular, given any set $A\subset[-\pi,\pi]$ (with $\mu(\partial A)=0$) we have that $\pi_p(x;A)$ is asymptotically the same as $\pi_p(x,\sigma A)$. We wish study the question of whether we can have an exact equality.
\begin{defn}
The holonomy angles are said to be invariant under the sign change $\sigma\in\{\pm 1\}^n$ if for any angle $\theta\in[-\pi,\pi]^n$, the number of primitive closed geodesics $C$ with $\theta_C=\theta$ is equal to the number of primitive closed geodesics $C'$ with $\theta_{C'}=\sigma\theta$.
\end{defn}

\begin{rem}
We note that if $C$ is a closed geodesic with length $l_C$ and holonomy angle $\theta_C$, then for any other closed geodesic, $l_{C'}=l_{C}$ if and only if $\theta_{C'}\in \{\sigma \theta_C|\sigma\in\{\pm 1\}^n\}$ (see Remark \ref{r:traces}). In particular, if the holonomy angles are invariant under all sign changes, then the number of closed geodesics with a fixed length is divisible by $2^n$.
\end{rem}

From a dynamical point of view it is not clear why this invariance should hold.  To further illustrate how remarkable this invariance is we relate it to a question of Selberg \cite{Selberg95}, which is, whether the spectrum of the space $V_m(\G\bs G)$, defined in (\ref{e:Vm}), changes when we change the sign of the $m_j$'s.
\begin{thm}\label{t:sign}
The holonomy angles are invariant under a sign change of $\sigma\in\{\pm1\}^n$, if and only if there is a spectral correspondence between $V_m(\G\bs G)$ and $V_{\sigma m}(\G\bs G)$ for all weights $m$. That is, there is a linear isomorphism  $\Theta_\sigma:V_m(\Gamma\bs G)\to V_{\sigma m}(\Gamma\bs G)$ that commutes with $\Omega_0$.
\end{thm}

For any geodesic $C$ denote by $C^{-1}$ its time reversal. We note that $\theta_{C^{-1}}=-\theta_C$, which implies that the primitive closed geodesics are invariant under $\sigma=(-1,\ldots,-1)$. Correspondingly, we have that complex conjugation sends $V_m$ to $V_{-m}$ and commutes with $\Omega_0$. For other signs $\sigma\in\{\pm 1\}^n$ it is not clear how to construct such a map and, as mentioned by Selberg in his lecture notes \cite{Selberg95}, whether such a correspondence even exists. It is thus interesting to find examples of lattices for which we can prove this correspondence. The following theorem gives many such examples.

\begin{thm}\label{t:sign2}
Let $K$ be a number field satisfying that its class number $h_K$ equals its narrow class number $h_K^+$.
Assume that $\Gamma$ is a lattice derived from a maximal order in a quaternion algebra defined over $K$, or a principal congruence group in such a lattice. Then, the holonomy angles are invariant under all sign changes.
\end{thm}
We recall that, by Margulis's Arithmeticity Theorem, any irreducible lattice in $\PSL(2,\bbR)^{n+1}$ is commensurable to a lattice derived form a maximal order in some quaternion algebra, and that this algebra is defined over the trace field of $\Gamma$ (that is, the field generated by traces of elements in $\Gamma$); see \cite{Margulis91}. We thus formulate the following conjecture:
\begin{conj}
For a lattice $\Gamma$ with trace field $K$ satisfying that $h_K=h_K^+$ the holonomy angles are invariant under all sign changes.
\end{conj}

We note that there are number fields satisfying that $h_K=h_K^+$ so that this condition is not empty.
For example, let $K_q$ denote the maximal real subfield of the cyclothymic field $\bbQ(e^{\frac{2\pi i}{q}})$, then $h_{K_q}=h_{K_q}^+$ whenever $q=2^a$ is a dyadic power (cf. \cite[Remark 3.4]{KimLim07}).
Also, when $q$ is a prime number such that $\tfrac{q-1}{2}$ is also prime, Taussky's conjecture implies that $h_{K_q}=h_{K_q}^+$ (see \cite{KimLim07} for some generalizations and results towards this conjecture).
On the other hand, Garbanati \cite{Garbanati76} showed that if $q$ is not a prime power, then $h_{K_q}\neq h_{K_q}^+$ so there are also many examples of number fields where this condition does not hold.

\subsection*{An interpretation for the holonomy angles}
We give another interpretation for the holonomy angles which explains their limiting distribution.
Assume that the lattice $\Gamma\subset\PSL(2,\bbR)^{n+1}$ is derived from a maximal order in a quaternion algebra.
When $n$ is even we can find a corresponding lattice $\tilde\Gamma\subset\PSL(2,\bbR)$ derived from a maximal order in a different quaternion algebra that is ramified in the same finite places and unramified in one infinite place. We then have a correspondence between closed geodesics on $\tilde \calM=\tilde\Gamma\bs\bbH$ and closed geodesics on $\Sigma_0$.
Moreover, the lattice $\tilde\Gamma$ comes equipped with a natural homomorphism $\rho:\tilde\Gamma\to \PSU(2)^n$ which enables us to attach to each closed geodesic $C$ on $\tilde \calM$ a conjugacy class $\{u_C\}$ in $\PSU(2)^n$ (see section \ref{s:geometric} for details). We then have the following correspondence:
\begin{thm}\label{t:Correspondence}
For each closed geodesic on $\Sigma_0$ there is a closed geodesic on $T^1\tilde\calM$ with the same length.
On the other hand, for each closed geodesic $C$ in $T^1\tilde \calM$ with a corresponding conjugacy class $\{u_C\}$ there are $2^n$ corresponding closed geodesics $\{C_i\}_{i=1}^{2^n}$ in $\Sigma_0$ with the same length $l_C=l_{C_i}$ and holonomy angles satisfying $|2\cos(\theta_{C_i}/2)|=|\Tr(u_C)|$.
\end{thm}
With this correspondence, Theorem \ref{t:equi} is equivalent to the equidistribution of the conjugacy classes $\{u_C\}$ as the lengths $l_C\to\infty$. We note that the equidistribution of these conjugacy classes can be proved (in a more general setting) via a different method using the Selberg trace formula on the space of vector valued functions on $\tilde{\Gamma}\bs \bbH$.

\begin{rem}
Theorem \ref{t:Correspondence} implies in particular that the number of primitive closed geodesics with a fixed length (and holonomy angle $\theta_C\in\{\sigma \theta:\sigma\in\{\pm 1\}^n\}$) is divisible by $2^n$. Note that for this theorem we do not assume that $h_K=h_K^+$. This suggests that the geodesics could be invariant under all sign changes even without this condition on the trace field.
\end{rem}

\subsection*{Outline of paper}
For the readers convenience we include a short outline of the paper.
In section \ref{s:Background}, we recall the essential background for the geodesic flow on $\calM=\Gamma\bs\calH^{(n)}$ and the relation between closed geodesics, holonomy, and conjugacy classes. In section \ref{s:trace} we recall the spectral theory on $\Gamma\bs G$ and in particular the Selberg trace formula.  In section \ref{s:Estimates}, we use the trace formula to prove Theorems \ref{t:PGT0} and \ref{t:equi}. The trace formula is also used in section \ref{s:sign} to prove Theorem \ref{t:sign}. In section \ref{s:quaternion}, we recall some results on lattices derived from quaternion algebras and prove Theorems \ref{t:sign2} and \ref{t:Correspondence}.
In section \ref{s:applications}, we give some additional applications. In particular, we prove a special case of the Jacquet-Langlands correspondence and we obtain some results about the distribution of fundamental units corresponding to certain quadratic orders.

\subsection*{Acknowledgements} I thank Zeev Rudnick for his comments.

\section{Background}\label{s:Background}
\subsection{The hyperbolic plane}
Let $\bbH=\set{x+iy| y>0}$ denote the upper half plane. For $z=x+iy\in \bbH$ we identify $T_z\bbH\cong\bbC$; with this identification the Riemannian metric is given by the inner product $\langle \xi,\eta\rangle_z=\frac{\Re(\xi \bar{\eta})}{y^2}$ on $T_z\bbH$.
There is a natural action of $\PSL(2,\bbR)=\SL(2,\bbR)/\pm I$ on $\bbH$ by isometries. That is
$z\mapsto g(z)=\frac{az+b}{cz+d}$, with the induced map on the tangent space given by
\[(z,\xi)\mapsto (g(z),g'(z)\xi),\]
where $g'(z)=\frac{1}{(cz+d)^2}$ denotes the differential of $g$.
After fixing the point $(i,i)\in T^1\bbH$ we can identify the unit tangent bundle with $\PSL(2,\bbR)$ via the action  $g\mapsto (g(i),g'(i)i)$.
For any $x,t\in\bbR$ and $\theta\in [-\pi,\pi]$ let
$$n_x=\begin{pmatrix} 1 & x \\ 0& 1\end{pmatrix},\;a_t=\begin{pmatrix} e^{t/2} & 0 \\ 0& e^{-t/2}\end{pmatrix},\;k_\theta=\begin{pmatrix} \cos(\tfrac \theta 2) & \sin(\tfrac \theta 2) \\ -\sin(\tfrac \theta 2) & \cos(\tfrac \theta 2) \end{pmatrix}.$$
For any $g\in \PSL(2,\bbR)$ there is a unique decomposition $g=p_zk_\theta$ with
$p_{z}=\begin{pmatrix}
\sqrt{y}  & x/ \sqrt{y}  \\
0 & 1/ \sqrt{y}
\end{pmatrix}=n_xa_{\ln(y)}$.
In these coordinates the identification of $\PSL(2,\bbR)$ with $T^1\bbH$ is given by $p_zk_\theta(i,i)=(z,e^{i\theta})$.

\subsection{Products of hyperbolic planes}
Let $\calH^{(n)}=\bbH_0\times\cdots\times\bbH_n$ denote a product of $n+1$ hyperbolic
planes endowed with the product metric. That is for $z=(z_0,\ldots,z_n)\in\calH^{(n)}$ the inner product on $T_z\calH^{(n)}\cong\bbC^{n+1}$ is given by
\[\langle\xi,\eta\rangle_{z}=\sum_{j=0}^{n}\langle \xi_j,\eta_j\rangle_{z_j}=\sum_{j=0}^{n}\frac{\Re(\xi_j \bar{\eta_j})}{y_j^2}.\]
There is a natural isometric action of $G=\PSL(2,\bbR)^{n+1}$ on $T\calH^{(n)}$ given by the $\PSL(2,\bbR)$ action on each factor. Each of the functions  $E_j(z,\xi)=\sqrt{\langle \xi_j,\xi_j\rangle_{z_j}}$ is invariant under the action of $G$ (and hence this action is not transitive on $T^1\calH^{(n)}$).
For any $\bfE=(E_0,\ldots,E_n)\in[0,\infty)^{n+1}$ let
$$T_\bfE\calH^{(n)}=\{(z,\xi)\in T\calH^{(n)}| E_j(z,\xi)=E_j\}\subset T\calM.$$
The group $G$ acts transitively on $T_{\bfE}\calH^{(n)}$ with the stabilizer of $(i,\bfE i)$ being $\prod_{E_j=0}K_j$; this give an identification of $T_{\bfE}\calH^{(n)}$ with the quotient $G/\prod_{E_j=0}K_j$.

For any $(z,\xi)\in T\calH^{(n)}$ there is a unique geodesic $C:\bbR\to \calH^{(n)}$ such that
$C(0)=z,\;C'(0)=\xi$. The geodesic flow $\phi^t:T\calH^{(n)}\to T\calH^{(n)}$ sends a point $(z,\xi)$ to the point obtained by flowing for time $t$ along this geodesic.
The spaces $T_\bfE\calH^{(n)}$ remain invariant under this flow. After identifying $T_\bfE\calH^{(n)}$ with $G/\prod_{E_j=0}K_j$ the geodesic flow $\phi^t$ is given by the right action of $a_{\bfE t}=(a_{E_0t},\ldots a_{E_nt})$, that is, if $(z,\xi)=g(i,\bfE i)$, then $\phi^t(z,\xi)=ga_{\bfE t}(i,\bfE i)$.

\subsection{Quotients}
Consider a quotient $\calM=\Gamma\bs\calH^{(n)}$ with $\Gamma$ an irreducible co-compact lattice in $G=\PSL(2,\bbR)^{n+1}$, that is, $\Gamma$ is a discrete subgroup such that the projection to any factor of $G$ is dense and that the quotient
$\G\bs G$ is compact. The geodesic flow on $\calM=\Gamma\bs\calH^{(n)}$ is given by the projection of the geodesic flow on $T\calH^{(n)}$ to $\calM$.
The energy shells
\[\Sigma(\bfE)=\{(x,\xi)\in T\calM| E_j(\xi)=E_j\}\]
are invariant under the geodesic flow and the identification of $T_\bfE\calH^{(n)}$ with $G/\prod_{E_j=0}K_j$ descends to an identification
$\Sigma(\bfE)\cong \G\bs G/\prod_{E_j=0}K_j$. From here on we fix $\bfE_0=(1,0,\ldots,0)$ and let $\Sigma_0=\Sigma(\bfE_0)$.
We denote by $K_0=\prod_{j=1}^n K_j$ and identify $\Sigma_0\cong \G\bs G/K_0$.

\subsection{Holonomy of closed geodesics on $\Sigma_0$}
Let $\calM=\G\bs\calH^{(n)}$ be as above and $\Sigma_0\subset T^1\calM$ denote the energy shell corresponding to $\bfE_0=(1,0,\ldots,0)$.
A closed geodesic of length $l$ on $\calM$ is given by an initial condition $(z,\xi)\in T^1\calM$ such that $\phi^l(z,\xi)=\gamma(z,\xi)$ for some $\gamma\in \Gamma$.
The geodesic is called primitive if $l$ is the smallest time where this holds. The geodesic lies in $\Sigma_0$ if the initial condition satisfies
$E_j(z,\xi)=0\;\forall  j\geq 1$ (that is, if $\xi=(\xi_0,0,\ldots,0)$).

Given any closed curve $C$ with initial condition $(z,\xi)\in T\calM$, parallel transporting tangent vectors about this curve gives the holonomy map $H_C:T_{z}\calM\to T_{z}\calM$. Under the identification $T_z\calM\cong \bbC^{n+1}$ the holonomy map is given by
$H_C(\xi_1,\ldots,\xi_n)=(\xi_0e^{i\theta_0},\ldots, \xi_ne^{i\theta_{n}})$. Changing the base point of the curve will give a conjugate element, and since the holonomy group is commutative $H_C$ does not depend on the base point. Thus, for any closed curve $C$ we attach an angle $\theta_C\in [-\pi,\pi]^{n+1}$ corresponding to the rotations of the holonomy map. Parallel transporting along a geodesic leaves the tangent to the geodesic invariant, hence,  there could be no nontrivial rotations in factors where the tangent vector is nonzero. Consequently, the holonomy map attached to a closed geodesic in $\Sigma_0$ satisfies $\theta_{0}=0$ and we let
$\theta_C\in [-\pi,\pi]^n$ be the angles corresponding to the last $n$ factors.

We have the following correspondence between closed geodesics in $\Sigma_0$ and hyperbolic-elliptic conjugacy classes in $\Gamma$, that is, conjugacy classes of elements that are hyperbolic in the first factor and elliptic in the last $n$ factors. With the identification $\Sigma_0\cong  G/K_0$, a closed geodesic of length $l$ is given by an initial condition $g(i,\bfE_0 i)\in T^1\calM$, with $g\in G$ satisfying that $$ga_{\bfE_0 l}=\gamma g\pmod{K_0}$$ for some $\gamma\in \Gamma$. That is, in the last $n$ factors we have
$g_j^{-1}\gamma_jg_j=k_{\theta_{\gamma,j}}\in K_j$ and in the first factor we have $g_0^{-1}\gamma_0g_0=a_{l_\gamma}$. If we change the initial condition $g$ to an equivalent one $\tau g$ for some $\tau\in \Gamma$ the resulting lattice element will be $\tau^{-1}\gamma\tau$. Hence, for any closed geodesic there is a corresponding hyperbolic-elliptic conjugacy class.  For the other direction, given any hyperbolic-elliptic $\gamma\in \Gamma$ let $g\in G$ such that $g^{-1}\gamma g=(a_{l_\gamma},k_{\theta_{\gamma,1}},\ldots,k_{\theta_{\gamma,n}})$.  Then $ga_{\bfE_0 l}=\gamma g\pmod{K_0}$, so that the initial condition $g(i,\bfE_0i)\in T^1\calM$ gives rise to a closed geodesic of length $l_\gamma$. The choice of such an element is not unique, however, if $\tilde g\in G$ is another element satisfying this, then $\tilde g=g(a_t,k_0)$ with $k_0\in K_0$, so that $\tilde g$ gives rise to the same closed geodesic (with a different starting point). Under this correspondence, a closed geodesic $C$ with length $l_C$ and holonomy angle $\theta_C$ corresponds to a conjugacy class $\{\gamma\}$ with $\theta_{\gamma,j}=\theta_{C,j}$ and $l_{\gamma}=l_C$.

\begin{rem}\label{r:traces}
Note that $\Tr(\gamma_j)=2\cos(\theta_{\gamma,j}/2)$ so that the trace in each factor is determined by the holonomy angle, and the holonomy angle is determined up to a sign by the trace. Moreover, we recall that $\Gamma$ is commensurable to a lattice derived from a quaternion algebra defined over some totally real number field.
Hence, the traces $\Tr(\gamma_j)$ for $j=0,\ldots,n$ are all different embeddings of the same field element into $\bbR$. Consequently, the length of the geodesic determines the holonomy angle up to a sign and the holonomy angle determine the length of the geodesic.
\end{rem}

\section{Trace formula}\label{s:trace}
The main ingredient in the proof of Theorems \ref{t:PGT0} and \ref{t:equi} is the Selberg Trace Formula. In particular, we use a hybrid version of it introduced by Selberg in \cite{Selberg95}.
We now recall some background on the spectral theory of $\PSL(2,\bbR)^{n+1}$ and outline the derivation of the hybrid trace formula.
\subsection{Spectral decomposition}
Let $G=\PSL(2,\bbR)^{n+1}$ and $\G\subset G$ an irreducible co-compact lattice.
Let $\rho:\Gamma\to U(N)$ be an irreducible unitary representation of $\Gamma$. Let $L^2(\Gamma\bs G,\rho)$ denote the space of vector valued functions on $G$ satisfying that $f(\gamma
g)=\rho(\gamma)f(g)$ for $\gamma\in \Gamma$ and that $\int_{\Gamma\bs G}|f(g)|^2dg<\infty$.
For any $m\in \bbZ^{n+1}$ let $\chi_m$ denote the characters of $K$ given by $\chi_m(k_\theta)=e^{i m\cdot \theta}$.
The group $K$ acts on $L^2(\G\bs G)$ (by right multiplication)
and induces a decomposition into the different $K$-types
\[L^2(\G \bs G,\rho)=\bigoplus_{m\in\bbZ^{n+1}} L^2(\G\bs G,\rho,m),\]
where $\psi\in L^2(\G\bs G,\rho,m)$ is of $K$-type $m$ if $\psi(gk)=\psi(g)\chi_m(k)$.

The algebra of invariant differential operators of $G$ is of rank
$n+1$ and is generated by operators $\Omega_0,\ldots,\Omega_n$ given
in the $z,\theta$ coordinates by
$$\Omega_j=y_j^2(\pd{^2}{x_j^2}+\pd{^2}{y_j^2})-2y_j\pd{}{\theta_j}\pd{}{x_j}$$
 (cf. \cite[Chapter X \S 2]{Lang85}).
These operators act on $L^2(\G\bs G,\rho,m)$ and we can further decompose it to a direct sum of joint eigenspaces
\[L^2(\G\bs G,\rho,m)=\bigoplus_{\lambda\in \Lambda(m)} \calC(m,\lambda),\]
where $\Lambda(m)\subset\bbR^n$ is the set of joint eigenvalues and $\calC(m,\lambda)$ the joint eigenspace corresponding to $\lambda=(\lambda_0,\ldots,\lambda_{n})$. We denote the dimension of the joint eigenspaces by $d(m,\lambda)=\dim C(m,\lambda)$.

The operators
$$E_j^{\pm}=\pm 2iy_je^{\pm i\theta_j}(\pd{}{x_j}\mp i\pd{}{y_j})\mp 2ie^{\pm i\theta_j}\pd{}{\theta_j},$$
commute with $\Omega_j$ and act as raising and lowering operators between the different
$K$-types. Moreover, on the space $\calC(m,\lambda)$ the operator $E_j^\pm E_j^\mp$ acts by multiplication by $4(\pm m_j(1\mp m_j))-\lambda_j)$,
so that $$E_j^{\mp}:\calC(m,\lambda)\to\calC(m\mp e_j,\lambda)$$
is a bijection if $\lambda_j\neq \pm m_j(1\mp m_j)$ and vanishes otherwise (see
\cite[Chapter VI \S 4,5]{Lang85}). We thus get the following characterization of the possible eigenvalues in $\Lambda(m)$:
\begin{prop}\label{p:Lambda}
For $\lambda\in \Lambda(0)$ either $\lambda=0$ or $\lambda\in(0,\infty)^{n+1}$.
If $m\neq 0$, then any $\lambda\in \Lambda(m)$ satisfies that
$$\lambda_j\in(0,\infty)\cup \{k(1-k)|k\in\bbN, k\leq |m_j|\}.$$
\end{prop}
\begin{proof}
Let $\lambda\in \Lambda(m)$ and let $\psi\in \calC(m,\lambda)$ denote the corresponding joint eigenfunction.
We first show that if $m_j=0$, then $\lambda_j\in(0,\infty)$.
For $m_j=0$ the operator $\Omega_j$  reduces to the hyperbolic Laplacian
$$\lap_j=y_j^2(\pd{^2}{x_j^2}+\pd{^2}{y_j^2}),$$
and the corresponding eigenvalue satisfies $\lambda_j\in[0,\infty)$. Moreover, if $\lambda_j=0$, then $\psi(z_0,\ldots,z_n)$ does not depend on $z_j$ and since $\Gamma$ is irreducible this implies that $\psi$ must be the constant function, which can only occur if $m=0$ (and $\rho$ is the trivial representation).

Next if $m_j>0$ (respectively $m_j<0$), if $\lambda_j\neq k(1-k)$ with $k\leq |m_j|$, then $(E^-)^{m_j}\psi$ (respectively $(E^+)^{|m_j|}\psi$) is in $\calC(\lambda,m-m_je_j)$. The argument for $m_j=0$ now implies that $\lambda_j\in(0,\infty)$.
\end{proof}
\subsection{Trace formula for $L^2(\G\bs G,\rho,m)$}
We recall the trace formula for $L^2(\G\bs G,\rho,m)$ for a weight $m\in \bbZ^{n+1}$ and representation $\rho$.
For $m=0$ and trivial $\rho$ the derivation of the trace formula is given in \cite[Chapter I]{Efrat87}.
For $n=0$ and arbitrary $m$ and $\rho$ it is described in \cite[Chapter 4]{Hejhal76}. For $n\geq 1$ and arbitrary $m$ (when $\Gamma$ is co-compact)
the trace formula takes the following form:\\
Let $h_j,\;j=0,\ldots,n$ be analytic functions with Fourier transforms $\hat{h}_j$ compactly supported. Then
\begin{eqnarray*}\lefteqn{\sum_{\lambda_k\in\Lambda(m)} d(m,\lambda_k)\left(\prod_j h(r_{k,j})\right)=}\\
&&\dim(\rho)\vol(\calF_{\Gamma})\prod_j \left(\frac{-1}{4\pi}\int_{-\infty}^\infty \frac{\hat{h}_j'(u)}{\sinh(u/2)}e^{-m_ju}du\right)\\
 && +\sum_{\{\gamma\}} \chi_\rho(\gamma)\vol(\Gamma_\gamma\bs G_\gamma)\left(\prod_{\gamma_j\sim a_{l_j}}\frac{\hat{h}_j(l_j)}{2\sinh(l_j/2)}\right)\left(\prod_{\gamma_j\sim k_{\theta_j}} \frac{\tilde{h}_{j}(\theta_j,m_j)}{\sin(\theta_j/2)}\right)
\end{eqnarray*}
where $\calF_{\Gamma}$ is a fundamental domain for $\Gamma\bs \calH^{(n)}$, $\chi_\rho(\gamma)=\Tr(\rho(\gamma))$ is the character of $\rho$, and
\begin{equation}\label{e:hjm}
\tilde{h}_{j}(\theta_j,m_j)=\frac{i}{4}\int_{-\infty}^\infty \hat{h}_j(u)e^{(m_j-\frac{1}{2})(u+i\theta_j)}\left[\frac{e^u-e^{i\theta_j}}{\cosh(u)-\cos(\theta)}\right]du.
\end{equation}
The derivation is a direct application of the arguments in \cite{Efrat87,Hejhal76,Hejhal83}. In particular for the explicit form of the $\tilde{h}_{j}(\theta_j,m_j)$ we refer to \cite[equation 6.30 on p. 394]{Hejhal83} and for the integral corresponding to the trivial conjugacy class see \cite[page 396]{Hejhal83}.

\subsection{Hybrid trace formula}
For any $m_1,\ldots,m_n\in\bbZ\setminus\{0\}$ we let $m=(0,m_1,\ldots,m_n)\in\bbZ^{n+1}$ and define $V_m=V_m(\G\bs G,\rho,m)$ by
\begin{equation}\label{e:Vm2}
V_m=\{\psi\in L^2(\G\bs G,\rho,m)|\forall j\geq 1,\; \Omega_j\psi+|m_j|(1-|m_j|)\psi=0\}.
\end{equation}
Let $\{\psi_k^{(m)}\}$ be an orthonormal basis for $V_m$ composed of eigenfunctions of $\Omega_0$ with eigenvalues $$0<\lambda_0^{(m)}\leq \lambda_1^{(m)}\leq\ldots,$$
we use the usual parameterizations of the eigenvalues $\lambda_k^{(m)}=\tfrac{1}{4}+{r_k(m)}^2$.
We note that the eigenfunctions $\psi_k^{(m)}$ are hybrid forms that behave like Maass forms in the first factor and like modular forms of weight $m_j$ on the $j$'th factor. We thus call the trace formula for these forms a hybrid trace formula.

For $m=(0,m_1,\ldots,m_n)$ as above, let
\begin{equation}\label{e:|m|}
|m|^*=\prod_{j=1}^n(2|m_j|-1),
\end{equation}
and define the function $H_m:(-\pi,\pi)^n\to \bbC$ by
\begin{equation}\label{e:Hm}
H_m(\theta)=\prod_{j=1}^n\frac{e^{im_j\theta_{j}}}{(1-e^{i\sgn(m_j)\theta_{j}})},
\end{equation}
where $\sgn(m_j)=\frac{m_j}{|m_j|}$.

Recall that an element $\gamma=(\gamma_0,\ldots,\gamma_n)\in \Gamma$ is called hyperbolic-elliptic if $|\Tr(\gamma_0)|>2$ and $|\Tr(\gamma_j)|<2$ for $j\geq 1$ and elliptic if $|\Tr(\gamma_j)|<2$ for all $j$. We note that an elliptic element $\gamma$ is always of finite order and we denote by $M_\gamma$ its order.  There could be at most finitely many elliptic conjugacy classes.
\begin{thm}\label{t:Htrace}
For any meromorphic function $h(r)$ with $\hat{h}$ compactly supported
\begin{eqnarray*}
\lefteqn{\sum_k h(r_{k}(m)) +(-1)^{n}\delta_{\rho,1}\delta_{m,\sgn(m)}h(i/2)=}\\&&  =\frac{\dim(\rho)|m|^*\vol(\calF_{\Gamma})}{(4\pi)^{n+1}}\int_{-\infty}^\infty h(r)r\tanh(\pi r)dr\\
&&+ {\sum_{\{\gamma\}}}^{'}\frac{\chi_\rho(\gamma)l_{\gamma_p}\hat{h}(l_\gamma)}{2\sinh(l_\gamma/2)}H_m(\theta_{\gamma})
+{\sum_{\{\gamma\}}}''\frac{\chi_\rho(\gamma)\tilde{h}(\theta_{\gamma_0},0)}{M_{\gamma_p}\sin(\theta_{\gamma_0}/2)}H_m(\theta_\gamma)\\
\end{eqnarray*}
where $\sum'$ is a sum over hyperbolic-elliptic conjugacy classes, $\sum''$ is a finite sum over the elliptic conjugacy classes, and $\gamma_p$ denotes the primitive element corresponding to $\gamma$.
\end{thm}
\begin{rem}
The sum over the hyperbolic-elliptic conjugacy class can be interpreted as a sum over all closed geodesics in $\Sigma_0$ where $l_\gamma=l_C$ is the length of the geodesic, $l_{\gamma_p}=l_{C_0}$ is the length of the primitive geodesic underlying $C$, and $\theta_\gamma=\theta_C$ are the holonomy angles attached to $C$.
\end{rem}

The function $H_m(\theta)$ appearing in the trace formula diverge as $\theta_j$ approaches zero.
However, after summing over all possible signs of $m$, the function
\[F_m(\theta)=(-1)^n\sum_{\sigma\in \{\pm 1\}^n} H_{\sigma m}(\theta)=\prod_{j=1}^n\frac{\sin((m-\tfrac{1}{2})\theta_j)}{\sin(\theta_j/2)},\]
is uniformly bounded and satisfies $|F_m(\theta)|\leq |m|^*$. Hence, it is convenient to add the corresponding trace formulas for all signs.
\begin{theorem}{\ref{t:Htrace}'}\label{c:HTF}
Let $m=(0,m_1,\ldots,m_n)$ with all $m_j\in\bbN$,
with the same notations as Theorem \ref{t:Htrace}
\begin{eqnarray*}
\lefteqn{\sum_{\sigma\in\{\pm1\}^n}\sum_{k} h(r_{k}(\sigma m))+\delta_{\rho,1}\delta_{m,1}(-2)^{n} h(i/2)=}\\
&& 2^n\frac{\dim(\rho)|m|^*\vol(\G\bs G)}{(4\pi)^{n+1}}\int_\bbR h(r)r\tanh(\pi r)dr+\\
&& +(-1)^{n}{\sum_{\{\gamma\}}}' \frac{\chi_\rho(\gamma)l_{\gamma_p} \hat{h}(l_\gamma)}{2\sinh(l_\gamma/2)}F_m(\theta_\gamma)
 +(-1)^n{\sum_{\{\gamma\}}}''\frac{\chi_\rho(\gamma)\tilde{h}(\theta_{\gamma_0},0)}{M_{\gamma_p}\sin(\theta_{\gamma_0}/2)}F_m(\theta_\gamma)\\
\end{eqnarray*}
\end{theorem}

\begin{rem}
In the case where the quotient $\Gamma\bs G$ is not compact a similar formula should still hold.
In order to derive the hybrid formula in the noncompact case it is possible to follow the same argument as below.
However, to do this one needs to compute the contribution of the continuous spectrum and the parabolic elements to the trace formula on $L^2(\Gamma\bs G,m)$. We leave this computation to a separate paper \cite{Kelmer09hybrid}.
\end{rem}

\subsection{Derivation of the Hybrid trace formula}
The hybrid trace formula was presented in Selberg's lecture notes \cite{Selberg95}. As its derivation does not seem to be written up anywhere in the literature we will include a proof here.
To simplify notation, we will write down the proof for the case where the representation $\rho$ of $\Gamma$ is the trivial representation, the modification needed for arbitrary $\rho$ are straightforward.

As described in \cite{Selberg95}, this formula can be derived by taking the trace of the following kernel
$$B(z,w)=f\left(\frac{|z_0-w_0|^2}{\Im(z_0)\Im(w_0)}\right)\prod_{j=1}^n \frac{(\Im(z_j)\Im(w_j))^{m_j}}{(z_j-\bar{w}_j)^{2m_j}},$$ with
$f$ a smooth compactly supported function such that $h$ is its Selberg transform. However, if $|m_j|=1$ for some $j$, then this kernel is not in $L^1(\calH^{(n)}\times\calH^{(n)})$ which introduces some difficulties with this approach. We will thus take a different approach by way of comparing the trace formulas on $L^2(\Gamma\bs G,m)$ for different $K$ types $m$.

For $m\in \bbZ^{n+1}$ let $\sgn(m)\in\{0,1,-1\}^{n+1}$ where $$\sgn(m_j)=\left\lbrace\begin{array}{cc}\frac{m_j}{|m_j|} & m_j\neq 0\\ 0 & m_j=0\end{array}\right..$$
Recall that $\Lambda(m)\subset\bbR^{n+1}$ denotes the set of joint eigenvalues on $L^2(\Gamma\bs G,m)$, that $\calC(m,\lambda)\subset L^2(\G\bs G,m)$ is the corresponding eigenspace and $d(m,\lambda)=\dim\calC(m,\lambda)$ is its dimension. For any set of indices $J\subset\{0,\ldots,n\}$ let
$$\Lambda(m,J)=\set{\lambda\in\Lambda(m)|\lambda_j=|m_j|(1-|m_j|),\;\forall j\in J}.$$
In particular, for $J=\{1,\ldots,n\}$ the eigenvalues in $\Lambda(m,J)$ counted with multiplicity are given by $\{\frac{1}{4}+r_k(m)^2\}_{k=0}^\infty$.

\begin{lem}
Let $m\in \bbZ^{n+1}$ and let $\sigma=\sgn(m)\in\{0,1,-1\}^{n+1}$. Let $J_m\subset\{0,\ldots,n\}$ denote the set of indices where $m_j\neq 0$. For any subset $J\subset J_m$ and any $j\in J$ with $m-\sigma_je_j\neq 0$ we have a disjoint union
\[\Lambda(m-\sigma_je_j,J\setminus\{j\})\cup\Lambda(m,J)=\Lambda(m,J\setminus\{j\}).\]
Also, $\forall j\in\{0,\ldots,n\}$ we have a disjoint union
\[\Lambda(0)\cup\Lambda(\pm e_j,\{j\})=\Lambda(\pm e_j)\cup\{0\}.\]
\end{lem}
\begin{proof}
Fix some $j\in J$ with $m-\sigma_je_j\neq 0$.
Since complex conjugation sends $\calC(\lambda,m)$ to $\calC(\lambda,-m)$ we may assume that $m_j\geq 0$ and $\sigma_j=1$. Recall that if $\lambda_j\neq  m_j(1-m_j)$, then $$E_j^{-}:\calC(m,\lambda)\to\calC(m- e_j,\lambda),\;E_j^{+}:\calC(m-e_j,\lambda)\to\calC(m,\lambda)$$
are bijections.

By proposition \ref{p:Lambda}, if $\lambda\in \Lambda(m-e_j)$, then $\lambda_j\neq m_j(1-m_j)$.
In particular, this implies that $\Lambda(m-\sigma_je_j,J\setminus\{j\})\cap\Lambda(m,J)=\emptyset$ so that their union is indeed disjoint.

Now let $\lambda\in \Lambda(m,J\setminus\{j\})$ and let $\psi\in\calC(m,\lambda)$ be a corresponding joint eigenfunction. Then either $\lambda_j=|m_j|(1-|m_j|)$ in which case
$\lambda\in \Lambda(m,J)$ or $\lambda_j\neq |m_j|(1-|m_j|)$ in which case $E^-\psi\in\calC(m-e_j,\lambda)$ so that $\lambda\in\Lambda(m-e_j,J\setminus\{j\})$.

To get inclusion in the other direction, we just need to show that $\Lambda(m-e_j,J\setminus\{j\})\subseteq \Lambda(m,J\setminus\{j\})$. Let $\lambda\in \Lambda(m-e_j,J\setminus\{j\})$ and let $\psi\in\calC(m-e_j,\lambda)$ be a corresponding eigenfunction. Since $\lambda\in \Lambda(m-e_j)$, then $\lambda_j\neq m_j(1-m_j)$ and hence $E^+\psi\in \calC(m,\lambda)$ implying that  $\lambda\in \Lambda(m,J\setminus\{j\})$.

Now for the case $m=0$. Any $0\neq\lambda\in\Lambda(0)$ satisfies that $\forall j,\;\lambda_j> 0$ and by applying $E_j^{\pm}$ to the corresponding eigenfunction we get that $\lambda\in \Lambda(\pm e_j)$; hence, $\Lambda(0)\subseteq \Lambda(\pm e_j)\cup\{0\}$. Next, for any $\lambda\in\Lambda(\pm e_j)$ either $\lambda_j=0$ or $\lambda_j\neq 0$ in which case $\lambda\in \Lambda(0)$; we thus get that $\Lambda(0)\cup\Lambda(\pm e_j,\{j\})=\Lambda(\pm e_j)\cup\{0\}$. Finally, since $0\not\in\Lambda(\pm e_j)$ and for any $0\neq\lambda\in\Lambda(0)$ we have that $\lambda_j\neq 0$ the union is disjoint.
\end{proof}

\begin{cor}\label{l:comb}
Let $m\in\bbZ^{n+1}$, $\sigma=\sgn(m)\in\{0,1,-1\}^{n+1}$, and $J_m$ be as above and
let $J\subseteq J_m$ be a subset.  If $m\neq\sigma$ or $J\neq J_m$, then for any function $\Psi$ on $\bbR^n$ for which the sum on the left absolutely converges we have
\[\sum_{\lambda\in \Lambda(m,J)}d(m,\lambda)\Psi(\lambda)=\sum_{J_\nu \subset J}(-1)^{|J_\nu|}\sum_{\lambda\in \Lambda(m-\sigma\nu)} d(m-\sigma\nu,\lambda)\Psi(\lambda),\]
while for $m=\sigma$ and $J=J_\sigma$ we have
\begin{eqnarray*}\lefteqn{\sum_{\lambda\in \Lambda(\sigma,J_\sigma)}d(\sigma,\lambda)\Psi(\lambda)+(-1)^{|J_\sigma|}\psi(0)=}\\ && \sum_{J_\nu \subset J_\sigma}(-1)^{|J_\nu|}\sum_{\lambda\in \Lambda(\sigma-\sigma\nu)} d(\sigma-\sigma\nu,\lambda)\Psi(\lambda),
\end{eqnarray*}
where the outer sums are over all $\nu\in\{0,1\}^{n+1}$ satisfying that $J_\nu\subseteq J$ with $J_\nu=\{j|\nu_j\neq 0\}$.
\end{cor}
\begin{proof}
The proof is a straight forward using induction on the size of $J$ and an inclusion exclusion argument. We omit the details.
\end{proof}
\begin{rem}
In the case where we the representation $\rho$ is not trivial, the trivial eigenvalue does not appear and we have the first equality also for $m=\sigma$ and $J=J_\sigma$.
\end{rem}

The derivation of the hybrid trace formula now follows from the above identity and the standard trace formula.
\begin{proof}[Proof of Theorem \ref{t:Htrace}]
Fix $m_1,\ldots,m_n\in\bbZ\setminus\{0\}$, let $m=(0,m_1,\ldots,m_n)$ and let $\sigma=\sgn(m)$ be as above. Apply the identity of Corollary \ref{l:comb} with $J=\{1,2,\ldots,n\}$ and $\Psi(\lambda)=\prod_{j=1}^{n+1} h(\sqrt{\lambda_j-\frac{1}{4}})$ to get
\begin{eqnarray*}\sum_{\lambda_k\in \Lambda(m,J)}d(m,\lambda)\Psi(\lambda)+(-1)^{n}\delta_{m,\sigma}\Psi(0)=\\
=\sum_{J_\nu \subset J}(-1)^{|J_\nu|}\sum_{\lambda\in \Lambda(m-\sigma\nu)} d(m-\sigma\nu,\lambda)\Psi(\lambda).
\end{eqnarray*}
For this choice of $\Psi$ and $J$ the sum on the left is given by
\[\left(\prod_{j=1}^n h(i(|m_j|-\tfrac{1}{2}))\right)\sum_{k}h(r_k(m)) +(-1)^{n}\delta_{m,\sigma}(h(i/2))^{n+1}.\]
Apply the trace formula to each of the inner sums on the right hand side.
The contribution of the trivial conjugacy class is
\[\sum_{J_\nu \subset J}(-1)^{|J_\nu|}\vol(\G\bs G) \prod_j \frac{-1}{2\pi}\int_{-\infty}^\infty \frac{\hat{h}'(u)}{2\sinh(u/2)}e^{-(m_j-\sigma_j\nu_j)u}du.\]
We can rewrite this sum as
\begin{eqnarray*}\lefteqn{\vol(\G\bs G) \left(\frac{-1}{2\pi}\int_{-\infty}^\infty \frac{\hat{h}'(u)}{2\sinh(u/2)}du\right)\times}\\&& \prod_{j=1}^n\left(\frac{-1}{2\pi}\int_{-\infty}^\infty \frac{\hat{h}'(u)}{2\sinh(u/2)}(e^{-m_ju}-e^{-(m_j-\sigma_j)u})du\right).
\end{eqnarray*}
Substitute for the first term
\begin{eqnarray*}\frac{-1}{2\pi}\int_{-\infty}^\infty \frac{\hat{h}'(u)}{2\sinh(u/2)}du&=& \frac{1}{\pi^2}\int_0^\infty h(r)r\int_0^\infty \frac{\sin(r u)}{2\sinh(u/2)}dudr\\
 &=& \frac{1}{4\pi}\int_{-\infty}^\infty h(r)r\tanh(\pi r)dr,
\end{eqnarray*}
and for each of the last $n$ terms
\begin{eqnarray*}\frac{-1}{2\pi}\int_{-\infty}^\infty \frac{\hat{h}'(u)}{2\sinh(u/2)}e^{-m_ju}(1-e^{\sigma_j u})du=
\frac{2|m_j|-1}{4\pi}h(i(|m_j|-\tfrac{1}{2})),\end{eqnarray*}
to get that the total contribution of the trivial conjugacy class is given by
\[\frac{|m|^*\vol(\G\bs G)}{(4\pi)^{n+1}}\int_{-\infty}^\infty h(r)r\tanh(\pi r)dr\prod_{j=1}^nh(i(|m_j|-\tfrac{1}{2})).\]
For the nontrivial conjugacy classes the contribution of a given class $\gamma=(\gamma_0,\ldots,\gamma_n)$ is given by
\[\vol(\G_\gamma\bs G_\gamma)\sum_{J_\nu \subset J}(-1)^{|J_\nu|}\left(\prod_{\gamma_j\sim a_{l_j}}\frac{\hat{h}(l_j)}{2\sinh(l_j/2)}\right)\left(\prod_{\gamma_j\sim k_{\theta_j}} \frac{\tilde{h}(\theta_j,m_j-\sigma_j\nu_j)}{\sin(\theta_j)}\right).\]
Since the contribution of the hyperbolic elements do not depend on $m_j$, this sum will vanish unless $\gamma_j\sim k_{\theta_j}$ is elliptic for all $j\neq 0$. In this case, if $\gamma_0\sim a_{l_\gamma}$ is hyperbolic the corresponding term is
\[\vol(\G_\gamma\bs G_\gamma )\frac{\hat{h}(l_\gamma)}{2\sinh(l_\gamma/2)}\prod_{j=1}^n
\left(\frac{\tilde{h}(\theta_{\gamma_j},m_j)}{\sin(\theta_{\gamma_j}/2)}-\frac{\tilde{h}(\theta_{\gamma_j},m_j-\sigma_j)}{\sin(\theta_{\gamma_j}/2)}\right),\]
and if $\gamma_0\sim k_{\theta_{\gamma_0}}$ is elliptic it is
\[\vol(\G_\gamma\bs G_\gamma )\frac{\tilde{h}(\theta_{\gamma_0},0)}{\sin(\theta_{\gamma_0}/2)}\prod_{j=1}^n
\left(\frac{\tilde{h}(\theta_{\gamma_j},m_j)}{\sin(\theta_{\gamma_j}/2)}-\frac{\tilde{h}(\theta_{\gamma_j},m_j-\sigma_j)}{\sin(\theta_{\gamma_j}/2)}\right).\]
A direct calculation using (\ref{e:hjm}) for the functions $\tilde{h}(\theta_j,m_j)$ yields
$$\frac{\tilde{h}(\theta_j,m_j)}{\sin(\theta_j/2)}-\frac{\tilde{h}(\theta_j,m_j-\sigma_j)}{\sin(\theta_j/2)}=\frac{e^{im_j\theta_j}}{(1-e^{i\sigma_j\theta_j})}h(i(|m_j|-\tfrac{1}{2})).$$ Hence, the contribution of each hyperbolic-elliptic conjugacy class is given by
\[\vol(\G_\gamma\bs G_\gamma )\frac{\hat{h}(l_\gamma)}{2\sinh(l_\gamma/2)}\prod_{j=1}^n\frac{e^{im_j\theta_{\gamma_j}}}{1-e^{i\sigma_j\theta_{\gamma_j}}}h(i(|m_j|-\frac{1}{2}))\]
and the contribution of each of the elliptic classes is given by
\[\vol(\G_\gamma\bs G_\gamma )\frac{\tilde{h}(\theta_{\gamma_0},0)}{\sin(\theta_{\gamma_0}/2)}\prod_{j=1}^n\frac{e^{im_j\theta_{\gamma_j}}}{1-e^{i\sigma_j\theta_{\gamma_j}}}h(i(|m_j|-\frac{1}{2})).\]

We thus get the following formula
\begin{eqnarray*}
\lefteqn{\left(\prod_{j=1}^n h(i(|m_j|-\tfrac{1}{2}))\right)\left(\sum_{k}h(r_k(m)) +(-1)^{n}\delta_{m,\sigma}(h(i/2))\right)=}\\
&& \frac{|m|^*\vol(\G\bs G)}{(4\pi)^{n+1}}\int_{-\infty}^\infty h(r)r\tanh(\pi r)dr\left(\prod_{j=1}^nh(i(|m_j|-\tfrac{1}{2}))\right)\\
&& +{\sum_{\{\gamma\}}}'\vol(\G_\gamma\bs G_\gamma )\frac{\hat{h}(l_\gamma)}{2\sinh(l_\gamma/2)}\prod_{j=1}^n\frac{e^{im_j\theta_{\gamma_j}}}{(1-e^{i\sigma_j\theta_j})}h(i(|m_j|-\tfrac{1}{2}))\\
&&+{\sum_{\{\gamma\}}}''\vol(\G_\gamma\bs G_\gamma )\frac{\tilde{h}(\theta_{\gamma_0},0)}{\sin(\theta_{\gamma_0}/2)}\prod_{j=1}^n\frac{e^{im_j\theta_{\gamma_j}}}{(1-e^{i\sigma_j\theta_j})}h(i(|m_j|-\tfrac{1}{2}))
\end{eqnarray*}

We recall that for any hyperbolic-elliptic $\gamma$ the group $\Gamma_\gamma$ is a cyclic group generated by some $\gamma_p$ (cf. \cite[Theorem 5.7]{Efrat87}) and that $\vol(\G_\gamma\bs G_\gamma )=l_{\gamma_p}$ (cf.\cite[Page 36]{Efrat87}). Also for any elliptic $\gamma\in\Gamma$ we have that $\Gamma_\gamma$ is a cyclic finite group generated by $\gamma_p$ and that $\vol(\G_\gamma\bs G_\gamma )=\frac{1}{M_{\gamma_p}}$.
Hence, after dividing both sides by $\prod_{j=1}^n h(i(|m_j|-\frac{1}{2}))$ we get the hybrid trace formula.
\end{proof}

%*****************************************************************************************************************************************
% *****************************************************************************************************************************************

\section{Asymptotic estimates}\label{s:Estimates}
In this section we prove Theorems \ref{t:PGT0} and \ref{t:equi} for the number of closed geodesics and the distribution of their holonomy angles. The proof uses the hybrid trace formula and the ideas of \cite{Sarnak83,SarnakWakayama99}.
\subsection{Preliminary estimates}
We start with some preliminary estimates for the spectrum and the number of closed geodesics.
\begin{lem}\label{l:evcount}
For $m\in \bbN^{n}$ and $x>1$ let
$$N_m(x)=\sum_{\sigma\in\{\pm1\}^n}\sharp\left\{k|r_{k}(\sigma m)\in [x-\tfrac{1}{2},x+\tfrac{1}{2}]\right\}.$$
Then  $N_m(x)=O(|m|^*x)$ as $x\to\infty$.
\end{lem}
\begin{proof}
Let $\Psi(t)$ be a smooth even function supported on $[-1,1]$ such that its Fourier transform $\hat{\Psi}$ is positive on $\bbR\cup i\bbR$ and satisfies that $\hat{\Psi}(0)=1$ and $\hat{\Psi}(r)>\tfrac{1}{2}$ for $|r|<\tfrac{1}{2}$.
For $x>0$ let $h_x(r)=\frac{\hat{\Psi}(r+x)+\hat{\Psi}(r-x)}{2}$ so that
\[N_m(x)\leq 2\sum_\sigma\sum_k h_x(r_k(\sigma m)).\]
Use the trace formula (Theorem \ref{t:Htrace}') with this test function to get
\begin{eqnarray*}
\lefteqn{\sum_\sigma\sum_k h_x(r_k(\sigma m)\leq\frac{|m|^*\vol(\G\bs G)}{(2\pi)^{n+1}}\int_\bbR h_x(r)r\tanh(\pi r)dr}\\
&&+|m|^*{\sum_{\{\gamma\}}}' \frac{l_{\gamma_p} |\hat{h}_x(l_\gamma)|}{\sinh(l_\gamma/2)}+|m|^*{\sum_{\{\gamma\}}}'' \frac{|\tilde{h}_x(\theta_{\gamma_0},0)|}{M_{\gamma_p}\sin(\theta_{\gamma_0}/2)}
\end{eqnarray*}
where we used the bound $|F_m(\theta)|\leq |m|^*$.
We can bound the integral
$|\int_\bbR h_x(r)r\tanh(\pi r)dr|\leq x\Psi(0)+\int_\bbR\hat\Psi(r)|r|dr=O(x)$. Next $|\hat{h}_x(t)|=|\Psi(t)|$ is bounded and vanishes outside of $[-1,1]$. Since there are only finitely many conjugacy classes with $l_\gamma<1$ the sum
$\sum' \frac{l_{\gamma_p} |\hat{h}_x(l_\gamma)|}{\sinh(l_\gamma/2)}|=O(1)$
is bounded uniformly in $x$.
To deal with the elliptic elements, using (\ref{e:hjm}), we can bound
$$|\tilde{h}_x(\theta,0)|\leq \int_{-\infty}^\infty \frac{|\hat{h}_x(u)|}{\cosh(u/2)-1}du\ll ||\hat{h}_x||_\infty,$$
where the notation $f(x)\ll g(x)$ mean that there is some constant $c>0$ such that $|f(x)|\leq cg(x)$.
Since $|\hat{h}_x(u)|=|\psi(u)|$ does not depend on $x$ the contribution of the elliptic terms are bounded by $O(1)$.
\end{proof}
\begin{rem}
The above proof with $x=0$ shows that $$\sharp\{k|r_k(m)\in i(0,\tfrac{1}{2}) \}=O(|m|^*).$$
\end{rem}

From this preliminary estimate we get the following estimate for the number of closed geodesics with length in some short interval.
\begin{lem}\label{l:shortsum}
For large $x>0$ and small $\epsilon>0$ we have
\[\sum_{x-\epsilon\leq l_\gamma\leq x+\epsilon} \frac{l_{\gamma_p}}{2\sinh(l_\gamma/2)}=O(\epsilon e^{x/2})+O(\epsilon^{-1}).\]
\end{lem}
\begin{proof}
Let $\Psi(t)$ be a positive smooth even function supported on $[-1,1]$ such that
$\int_{-1}^1 \Psi(t)dt=1$ and $\Psi(t)>\frac{1}{2}$ for $|t|\leq \frac{1}{2}$.
Use Theorem \ref{t:Htrace}' with $m=(1,1,\ldots,1)$ (so that $F_m(\theta)\equiv 1$) and test function $h_{x,\epsilon}$ with Fourier transform given by $\hat{h}_{x,\epsilon}(t)=\Psi(\frac{t-x}{2\epsilon})+\Psi(\frac{x+t}{2\epsilon})$. That is $h_{x,\epsilon}(r)=4\epsilon \hat{\Psi}(2\epsilon r)\cos(rx)$.
For any $t\in (x-\epsilon,x+\epsilon),\;h_{x,\epsilon}(t)\geq \frac{1}{2}$, hence
\[\sum_{x-\epsilon\leq l_\gamma\leq x+\epsilon} \frac{l_{\gamma_p}}{2\sinh(l_\gamma/2)}\leq 2{\sum_{\{\gamma\}}}' \frac{l_{\gamma_p}\hat{h}_{x,\epsilon}(l_\gamma)}{2\sinh(l_\gamma/2)}.\]
The sum on the right is bounded by
\begin{eqnarray*}
\lefteqn{|2^{n}h_{x,\epsilon}(i/2)|+\sum_\sigma\sum_{k}h_{x,\epsilon}(r_k(\sigma))}\\
&&+{\sum_{\{\gamma\}}}'' \frac{|\tilde{h}_{x,\epsilon}(\theta_{\gamma_0},0)|}{M_{\gamma_p}|\sinh(\theta_{\gamma_0})|}+\frac{\vol(\G\bs G)}{2(2\pi)^{n+1}}\int_\bbR h_{x,\epsilon}(r)r\tanh(\pi r)dr. \end{eqnarray*}

The contribution of the first term and all the complementary spectrum is bounded by some constant times
$$h_{x,\epsilon}(i/2)=2\epsilon \hat{\Psi}(\epsilon i)(e^{x/2}+e^{-x/2})=O(\epsilon e^{x/2}).$$
As before, since $||\hat{h}_{x,\epsilon}(t)||_\infty\leq 2||\Psi||_\infty$ is uniformly bounded, the contribution of the elliptic terms is bounded by $O(1)$.
We can bound the integral
\[\int_\bbR h_{x,\epsilon}(r)r\tanh(\pi r)dr\ll\epsilon\int_0^\infty \hat\Psi(\epsilon r)rdr=O(\epsilon^{-1}).\]
Finally, the contribution of the principal spectrum is bounded (up to some constant) by $\epsilon \sum_{r_k\in \bbR}\hat{\Psi}(2\epsilon r_k)=O(\frac{1}{\epsilon}),$
where we used the fast decay of $\hat{\Psi}(r)$ together with the estimate
\[\#\left\{k,\sigma|x-\tfrac{1}{2}\leq \epsilon r_k(\sigma)\leq x+\tfrac{1}{2} \right\}=O(\frac{x}{\epsilon^2}),\]
implied by lemma \ref{l:evcount}.
\end{proof}

\subsection{Main estimate}
The family of functions
$$\{H_m(\theta)|m_1,\ldots,m_n\in\bbZ\setminus\{0\}\},$$
with $H_m(\theta)$ given by (\ref{e:Hm}) form an orthonormal basis for $L^2((\bbR/2\pi\bbZ)^{n},\mu)$.
We can decompose any such function as $$f(\theta)=\sum_{m} a_f(m)H_m(\theta),$$
with $a_f(m)=\int f(\theta)\overline{H_m(\theta)}d\mu(\theta)$.
(The coefficients $a_f(m)$ are just the standard Fourier coefficients $\hat{g}(m)$ for $g(\theta)=f(\theta)\prod_j(1-e^{-\sgn(m_j)\theta_j})$.)
For any $f\in C^\infty((\bbR/2\pi\bbZ)^n)$ let
\begin{equation}\label{e:Cf}
C(f)=\sum_{m}|m|^*|a_f(m)|.
\end{equation}

\begin{prop}\label{p:main}
For any $f\in C^\infty((\bbR/2\pi\bbZ)^{n})$ we have
\[|{\mathop{\sum_{\{\gamma\}}}_{l_\gamma\leq x}}'\frac{l_{\gamma_p}f(\theta_\gamma)}{2\sinh(l_\gamma/2)}-2^{n+1}e^{x/2}\mu(f)|\ll \sqrt{\norm{f}_\infty C(f)}e^{x/4}+C(f) e^{\alpha x}.\]
where $\alpha=\sup_{m}\alpha_m$ as in Theorem \ref{t:equi}.
\end{prop}
\begin{proof}
Let $\Psi$ be a smooth positive even function supported on $[-1,1]$ and let $\id_x$ denote the indicator function of $[-x,x]$.
Let $\Psi_\epsilon(t)=\frac{1}{\epsilon}\Psi(t/\epsilon)$ and let $\hat{h}_{x,\epsilon}=\id_x*\Psi_\epsilon$ denote the convolution of $\id_x$ and $\Psi_\epsilon$. The function $\hat{h}_{x,\epsilon}$ is supported on $[-x-\epsilon,x+\epsilon]$, it takes values between zero and one and it is equal to one on $[-x+\epsilon,x-\epsilon]$.
Its inverse Fourier transform is given by $h_{x,\epsilon}(r)=\hat{\id}_{x}(r)\hat{\Psi}_\epsilon(r)=\frac{2\sin(xr)}{r}\hat{\Psi}(\epsilon r)$.
Using Lemma \ref{l:shortsum} we can approximate the sharp cutoff with the smooth one, that is,
\begin{eqnarray}\label{e:smoothing1}
|{\mathop{\sum_{\{\gamma\}}}_{l_\gamma\leq x}}'\frac{l_{\gamma_p}f(\theta_\gamma)}{2\sinh(l_\gamma/2)}-{\sum_{\{\gamma\}}}'\frac{l_{\gamma_p}h_{x,\epsilon}(l_\gamma)f(\theta_\gamma)}{2\sinh(l_\gamma/2)}|\leq
\\
\nonumber\leq \mathop{{\sum_{\{\gamma\}}}'}_{x-\epsilon\leq l_\gamma\leq x+\epsilon}\frac{l_{\gamma_p}|f(\theta_\gamma)|}{2\sinh(l_\gamma/2)}\ll \norm{f}_\infty(\epsilon e^{x/2}+\frac{1}{\epsilon}).
\end{eqnarray}
To estimate the smoothed sum, expand $f(\theta)=\sum_{m} a_f(m)H_m(\theta)$ to get
\begin{eqnarray*}
\lefteqn{|{\sum_{\{\gamma\}}}'\frac{l_{\gamma_p}h_{x,\epsilon}(l_\gamma)f(\theta_\gamma)}{2\sinh(l_\gamma/2)}-2^{n+1}e^{x/2}\mu(f)|\leq}\\
&&\sum_m |a_f(m)|\left|{\sum_{\{\gamma\}}}'\frac{l_{\gamma_p}h_{x,\epsilon}(l_\gamma)H_m(\theta_\gamma)}{2\sinh(l_\gamma/2)}-2^{n+1}e^{x/2}\mu(H_m)\right|.
\end{eqnarray*}
Using the trace formula we can replace each of the inner sums,
\begin{eqnarray*}	
\lefteqn{{\sum_{\{\gamma\}}}'\frac{l_{\gamma_p}h_{x,\epsilon}(l_\gamma)}{2\sinh(l_\gamma/2)}H_m(\theta_\gamma)=}\\
&=& (-1)^{n}\delta_{m,\sigma}h_{x,\epsilon}(i/2)+\sum_{k}h_{x,\epsilon}(r_k(m))\\
&-& {\sum_{\{\gamma\}}}''\frac{\tilde{h}_{x,\epsilon}(\theta_{\gamma_0},0)}{M_{\gamma_p}\sin(\theta_{\gamma_0}/2)}H_m(\theta_\gamma)- \frac{\vol(\G\bs G)|m|^*}{(4\pi)^{n+1}}\int_\bbR h_{x,\epsilon}(r)r\tanh(\pi r)dr.
\end{eqnarray*}
Estimate $h_{x,\epsilon}(i/2)=2e^{x/2}+O(\epsilon e^{x/2})$ and bound the contribution of the complementary spectrum by $O(|m|^*e^{\alpha x})$. The function $\hat{h}_{x,\epsilon}\leq 1$ is uniformly bounded so the contribution of the elliptic conjugacy classes are bounded by $O(1)$.
As in the proof of Lemma \ref{l:shortsum}, the contribution of the principal spectrum and the trivial conjugacy class are bounded by $O(|m|^*\epsilon^{-1})$. We thus get that
\[{\sum_{\{\gamma\}}}'\frac{l_{\gamma_p}h_{x,\epsilon}(l_\gamma)H_m(\theta_\gamma)}{2\sinh(l_\gamma/2)}=
(-1)^{n}2\delta_{m,\sigma}e^{x/2}(1+O(\epsilon))+O(|m|^*e^{\alpha x})+O(\frac{|m|^*}{\epsilon}).\]
Since $\mu(H_m)=(-\tfrac{1}{2})^{n}\delta_{m,\sigma}$ we get
\[\left|{\sum_{\{\gamma\}}}'\frac{l_{\gamma_p}h_{x,\epsilon}(l_\gamma)H_m(\theta_\gamma)}{2\sinh(l_\gamma/2)}-2^{n+1}e^{x/2}\mu(H_m)\right|\ll \frac{|m|^*}{\epsilon}+|m|^*e^{\alpha x}+\delta_{m,\sigma}\epsilon e^{x/2}.\]
Multiplying by $a_f(m)$ and summing over all $m$ we get
\begin{eqnarray*}
|{\sum_{\{\gamma\}}}'\frac{l_{\gamma_p}h_{x,\epsilon}(l_\gamma)f(\theta_\gamma)}{2\sinh(l_\gamma/2)}-2^{n+1}e^{x/2}\mu(f)|\ll C(f)e^{\alpha x}+C(f)\epsilon^{-1}+\norm{f}_\infty \epsilon e^{x/2},
\end{eqnarray*}
where we used the bound $\sum_{\sigma}|a_f(\sigma)|\ll \norm{f}_\infty$.
By (\ref{e:smoothing1}) we get the same estimate for the sharp cutoff
and taking $\epsilon=\sqrt{\frac{C(f)}{\norm{f}_\infty}}e^{-x/4}$ concludes the proof.
\end{proof}

\subsection{Proof of Theorem \ref{t:PGT0}}
The prime geodesic theorem now follows form Proposition \ref{p:main} with $f$ the constant function.
\begin{proof}
Apply Proposition \ref{p:main} with $f\equiv 1$ the constant function. We thus get
\begin{equation}\label{e:sum1}
{\mathop{\sum_{\{\gamma\}}}_{l_\gamma\leq x}}'\frac{l_{\gamma_p}}{2\sinh(l_\gamma/2)}=2^{n+1}e^{x/2}+O(e^{x/4})+O(e^{\alpha_1 x}).
\end{equation}
(The decomposition of the constant function only involves the function $H_\sigma$ with $\sigma\in\{\pm1\}^n$, hence, the error term depends only on $\alpha_1$.)
The contribution of the primitive conjugacy classes is trivially bounded by the whole sum so
\[\Theta(x):={\mathop{\sum_{\{\gamma_p\}}}_{l_{\gamma_p}\leq x}}'\frac{l_{\gamma_p}}{2\sinh(l_{\gamma_p}/2)}=O(e^{x/2}).\]
Integrating by parts we get the preliminary bound
$$\pi_p(x)=\int^x \frac{2\sinh(t/2)}{t}d\Theta(t)=O(\frac{e^x}{x}).$$
We can now estimate the contribution of the non-primitive conjugacy classes to (\ref{e:sum1}), that is
\[\sum_{k=2}^\infty\sum_{l_{\gamma_p}\leq x/k}\frac{l_{\gamma_p}}{2\sinh(kl_{\gamma_p}/2)}=\sum_{k=2}^{\infty} \int_\delta^{x/k}\frac{t}{2\sinh(kt/2)}d\pi_p(t).\]
Since there is a geodesic with a shortest length $\delta$ this is a finite sum. Using the preliminary bound $\pi_p(x)=O(\frac{e^x}{x})$ and integrating by parts we get that this sum is bounded by $O(x)$.
Consequently, the same asymptotic formula as (\ref{e:sum1}) is valid when summing over primitive geodesics, that is,
\[\Theta(x)=\sum_{l_{\gamma_p}\leq x}\frac{l_{\gamma_p}}{2\sinh(l_{\gamma_p}/2)}=2^{n+1}e^{x/2}+O(e^{x/4})+O(e^{\alpha_1 x})\]
Let $$\tilde\Theta(x)=\sum_{l_{\gamma_p}\leq x}l_{\gamma_p}e^{-l_{\gamma_p}/2},$$
then $\Theta(x)=\tilde\Theta(x)+O(1)$.
Using integration by parts together with the estimate
$\tilde\Theta(x)=2^{n+1}e^{x/2}+O(e^{x/4})+O(e^{\alpha x})$,
we get
\[\pi_p(x)=\int^x \frac{e^{t/2}}{t}d\tilde\Theta(t)=2^{n}Li(e^x)+O(\frac{e^{3 x/4}}{x})+O(e^{(\alpha_1+\tfrac12)x}).\]
\end{proof}

\subsection{Proof of Theorem \ref{t:equi}}
Similar to the previous case, the equidistribution for smooth test functions follows from Proposition \ref{p:main} and integration by parts.
\begin{proof}
Without loss of generality we may assume that $\int fd\mu=0$. Then, Proposition \ref{p:main} gives
\[|{\mathop{\sum_{\{\gamma\}}}_{l_\gamma\leq x}}'\frac{l_{\gamma_p}}{2\sinh(l_\gamma/2)}f(\theta_\gamma)|\ll \sqrt{\norm{f}_\infty C(f)}e^{x/4}+C(f) e^{\alpha x}.\]

The contribution of the non-primitive conjugacy classes is bounded by $O(x)$, so the same estimate holds when summing over primitive elements
\[{\mathop{\sum_{\{\gamma_p\}}}_{l_{\gamma_p}\leq x}}'\frac{l_{\gamma_p}}{2\sinh(l_{\gamma_p}/2)}f(\theta_{\gamma_p})|\ll \sqrt{\norm{f}_\infty C(f)}e^{x/4}+C(f) e^{\alpha x}.\]
As before, the same bound also holds for
\[\Theta_f(x):=\sum_{l_{\gamma_p}\leq x}e^{-l_{\gamma_p}/2}l_{\gamma_p}f(\theta_{\gamma_p})\ll \sqrt{\norm{f}_\infty C(f)}e^{x/4}+C(f) e^{\alpha x}.\]
Integrating by parts we get
\[\sum_{l_{\gamma_p}\leq x}f(\theta_{\gamma_p})=\int^x \frac{e^{t/2}}{t}d\Theta_f(t)\ll |\Theta_f(x)|e^{x/2},\]
concluding the proof.
\end{proof}

\subsection{A sharp cutoff}
Since we have good control on how the error term depends on $f$, we can give bounds for the error term also for a sharp cutoff function. We show this in the simple case where $f=\id_A$ is the indicator function of a rectangle.
\begin{cor}\label{c:equi}
For any rectangle $A\subset [-\pi,\pi]^{n}$ we have
\[\frac{\pi_p(x;A)}{\pi_p(x)}=\mu(A)+O(e^{-c x}),\]
where the implied constant is independent of $A$ and the exponent is given by
$c=\left\lbrace\begin{array}{cc} \frac{1}{2(n+2)} & \alpha\leq \frac{1}{2(n+2)} \\ \frac{1-2\alpha}{2(n+1)} & \alpha>\frac{1}{2(n+2)} \end{array}\right..$
\end{cor}
\begin{proof}
We will use the Selberg-Beurling functions to approximate the indicator function (see \cite[Chapter 1.2]{Montgomery94} for details). For any interval $\calJ\subset[-\pi,\pi]$ and any $N\in\bbN$ the corresponding Selberg-Beurling functions $S_{\calJ,N}^{\pm}(\theta)$ are trigonometric polynomials of degree at most $N$,
$$S_{\calJ,N}^{\pm}(\theta)=\sum_{|k|\leq N}\hat{S}_{\calJ,N}^{\pm}(k) e^{-ik\theta},$$
satisfying that
\begin{equation}\label{e:SB1}
S_{\calJ,N}^{-}(\theta)\leq \id_{\calJ}(\theta)\leq S_N^{+}(\theta),\quad \forall \theta \in[-\pi,\pi],
\end{equation}
\begin{equation}\label{e:SB2}
\int_{-\pi}^\pi S_{\calJ,N}^{\pm}(\theta)d\theta=|\calJ|\pm \frac{2\pi}{N+1},
\end{equation}
which implies that the non-zero Fourier coefficients satisfy
\begin{equation}\label{e:SB3}
|\hat{S}_{\calJ,N}^\pm(k)|\leq \frac{1}{N+1}+\frac{1}{|k|},\quad \forall 1\leq |k|\leq N.
\end{equation}

For a rectangle $A=\prod_{j}\calJ_j$ let $f_N(\theta)=\prod_{j}S^+_{\calJ_j,N}(\theta_j)$, then by (\ref{e:SB1}) we have $f_N(\theta)\geq \id_A(\theta)$ for all $\theta\in[-\pi,\pi]^n$ (note that these functions are non-negative). The coefficients $a_{f_N}(m)$ are given in terms of the Fourier coefficients
\[a_{f_N}(m)=\prod_{j=1}^n\left(\frac{\hat{S}^+_{\calJ_j,N}(m_j)-\hat{S}^+_{\calJ_j,N}(m_j-\sgn(m_j))}{2}\right).\]
Using (\ref{e:SB3}) we get that
$$C(f_N)=\sum_{m}|m|^* |a_{f_N}(m)|\leq (6N)^n.$$
A similar calculation gives that $|\mu(f_N)-\mu(\id_A)|\leq \frac{n}{N+1}$, and since $\norm{S^{+}_{\calJ,N}}_{\infty}\leq 5$ we get that $\norm{f_N}_{\infty}=O(1)$ is uniformly bounded.

Now use $f_N$ to get an upper bound for $\frac{\pi(x;A)}{\pi_p(x)}$,
\[\frac{\pi(x;A)}{\pi_p(x)}\leq \frac{1}{\pi_p(x)}{\sum_{l_{\gamma_p}\leq x}}'f_N(\theta_{\gamma_p}).\]
From Theorem \ref{t:equi}, taking into account how the error terms depend on $f_N$, we get
\begin{eqnarray*} \frac{1}{\pi_p(x)}{\sum_{l_{\gamma_p}\leq x}}'f_N(\theta_{\gamma_p})
&=& \mu(f_N)+O(\sqrt{C(f_N)}e^{-x/4})+O(C(f_N)e^{(\alpha-\tfrac{1}{2})x})\\
&=& \mu(A)+ O(N^{-1})+O(N^{n/2}e^{-x/4})+O( N^n e^{(\alpha-\tfrac{1}{2})x}).
\end{eqnarray*}
Let $N=e^{cx}$ where $c=\frac{1}{2(n+2)}$ if $\alpha\leq \frac{1}{2(n+2)}$ and $c=\frac{1-2\alpha}{2(n+1)}$ if $\alpha > \frac{1}{2(n+2)}$ to get the upper bound
\[\frac{\pi(x;A)}{\pi_p(x)}\leq \mu(A)+O(e^{-cx}).\]

In order to get a lower bound, we write $[-\pi,\pi]^n=A\cup A_1\cup\cdots\cup A_k$ as a finite disjoint union of rectangles (with $k\leq 2^n$). Since the error term does not depend on the rectangle $A_i$ we get
\[\frac{\pi(x;A)}{\pi_p(x)}=1-\sum_{i=1}^k\frac{\pi(x;A_i)}{\pi_p(x)}\geq 1-\sum_{i=1}^k\mu(A_i)+O(e^{-cx})=\mu(A)+O(e^{-cx}).\]
\end{proof}

%*****************************************************************************************************************************************
% ****
\section{Invariance of holonomy under sign changes}\label{s:sign}
Recall that the holonomy is invariant under the sign change $\sigma\in\{\pm 1\}^n$ if for any angle $\theta\in[-\pi,\pi]^n$ the number of primitive closed geodesics with holonomy $\theta$ is equal the number of closed geodesics with holonomy $\sigma\theta$. This property can be translated to the following property of the lattice: For every $\theta\in [-\pi,\pi]^n$ the number of primitive hyperbolic-elliptic $\Gamma$-conjugacy classes that are conjugated in $\PSL(2,\bbR)^{n+1}$ to $(a_l,k_\theta)$ equals the number of primitive hyperbolic-elliptic $\Gamma$-conjugacy classes that are conjugated in  $\PSL(2,\bbR)^{n+1}$ to $(a_l,k_{\sigma\theta})$.

We now use the trace formula to prove Theorem \ref{t:sign}, that is, we show that invariance of the holonomy angles under the sign change $\sigma$ is equivalent to a spectral correspondence between the spaces $V_m$ and $V_{\sigma m}$.
\begin{proof}[Proof of Theorem \ref{t:sign}]
Assume that the holonomy angles are invariant under the sign change $\sigma\in\{\pm1\}^n$.
To show that there is a linear map $\Theta_\sigma:V_m\to V_{\sigma m}$ commuting with $\Omega_0$ we just need to show that the spectrum of $\Omega_0$ is the same on both spaces, that is, that $\{r_k(m)\}_{k=0}^\infty=\{r_k(\sigma m)\}_{k=0}^\infty$.

Fix $m\in\bbZ^{n+1}$ with $m_0=0$ and $m_j\neq 0$ for $j\geq 1$. Let $\hat{h}\in C^\infty_c(\bbR)$ denote a smooth compactly supported function and consider the hyperbolic-elliptic sum in the trace formula, that is,
$${\sum_{\{\gamma\}}}'\frac{l_{\gamma_p}\hat{h}(l_\gamma)}{2\sinh(l_\gamma/2)}H_m(\theta_{\gamma}).$$
Since any conjugacy class $\{\gamma\}$ is a power of a primitive one we can rewrite this as
$${\sum_{\{\gamma_p\}}}'\sum_{q=1}^\infty\frac{l_{\gamma_p}\hat{h}(ql_{\gamma_p})}{2\sinh(ql_{\gamma_p}/2)}H_m(q\theta_{\gamma_p}).$$
Since the holonomy angles are invariant under the sign change $\sigma$ this is equal to
$${\sum_{\{\gamma_p\}}}'\sum_{q=1}^\infty\frac{l_{\gamma_p}\hat{h}(ql_{\gamma_p})}{2\sinh(ql_{\gamma_p}/2)}H_m(q\sigma\theta_{\gamma_p}).$$
We thus get the equality
\[{\sum_{\{\gamma\}}}'\frac{l_{\gamma_p}\hat{h}(l_\gamma)}{2\sinh(l_\gamma/2)}H_m(\theta_{\gamma})=
{\sum_{\{\gamma\}}}'\frac{l_{\gamma_p}\hat{h}(l_\gamma)}{2\sinh(l_\gamma/2)}H_m(\sigma\theta_{\gamma}).\]
Notice that $H_m(\sigma\theta)=H_{\sigma m}(\theta)$ so that we have
\[{\sum_{\{\gamma\}}}'\frac{l_{\gamma_p}\hat{h}(l_\gamma)}{2\sinh(l_\gamma/2)}H_m(\theta_{\gamma})=
{\sum_{\{\gamma\}}}'\frac{l_{\gamma_p}\hat{h}(l_\gamma)}{2\sinh(l_\gamma/2)}H_{\sigma m}(\theta_{\gamma}).\]
Since the contribution of the trivial conjugacy class to the trace formula does not depend on the sign of $m$,
after subtracting the contribution of the elliptic elements we get the equality
\begin{eqnarray}\label{e:spectral2}
\lefteqn{\sum_k h(r_k(m))- {\sum_{\{\gamma\}}}''\frac{\tilde{h}(\theta_{\gamma_0},0)H_m(\theta_{\gamma})}{M_{\gamma_p}\sin(\theta_{\gamma_0}/2)}=} \\
\nonumber &&=\sum_k h(r_k(\sigma m)) -{\sum_{\{\gamma\}}}''\frac{\tilde{h}(\theta_{\gamma_0},0)H_m(\sigma\theta_{\gamma})}{M_{\gamma_p}\sin(\theta_{\gamma_0}/2)}.
\end{eqnarray}

We claim that this equality implies that the spectrum of $V_m$ and $V_{\sigma m}$ are the same.
We first show that the exceptional spectrum is the same. Let
$$\tfrac{1}{2}\geq c_0(m)>c_1(m)\cdots>c_q(m)>0$$ be such that $\frac{1}{4}-c_k(m)^2\in[0,\frac{1}{4})$ are all the exceptional eigenvalues and let $d_k(m)$ denote the multiplicity of the corresponding eigenvalue.  Let
$\Psi$ denote a positive smooth even function supported on $[-1,1]$ with $\int_{-1}^1\Psi(t)dt=1$. For $x>0$ let $h_x(r)=\frac{2\sin(xr)}{r}\hat{\Psi}(r)$. For $r\in\bbR$ real we have that $|h_x(r)|\leq 2x|\hat\Psi(r)|$ so that the sum $|\sum_{r_k(m)\in\bbR} h_x(r_k(m))|=O(x)$. The Fourier transform $|\hat{h}_x(t)|=|\id_x*\Psi(t)|\leq 1 $ is uniformly bounded so the contribution of the elliptic elements are bounded by $O(1)$. Consequently, the equality (\ref{e:spectral2}) with this test function implies that
\[\sum_{k=0}^q d_k(m)\frac{e^{xc_k(m)}}{c_k(m)}\hat{\Psi}(ic_k(m))=\sum_{k=0}^q d_k(\sigma m)\frac{e^{xc_k(\sigma m)}}{c_k(\sigma m)}\hat{\Psi}(ic_k(m)) +O(x).\]
Dividing by $e^{xc_0(m)}$ and taking $x\to\infty$ implies that $c_0(m)=c_0(\sigma m)$ and that $d_0(m)=d_0(\sigma m)$.
We can thus subtract its contribution form both sides and continue in the same way to get that $c_k(m)=c_k(\sigma m)$ and $d_k(m)=d_k(\sigma m)$ for all $k=1,\ldots,q$.

Now for the principal spectrum. Let $h$ be an even function with Fourier transform $\hat{h}\in C^\infty(\bbR)$ even, supported on $[-1,1]$, and satisfies $\int_{-1}^1\hat{h}(t)dt=1$.  For $x,\epsilon>0$ let $h_{\epsilon,x}(r)=\tfrac{h(\frac{x-r}{\epsilon})+h(\frac{x+r}{\epsilon})}{2}$ so that its Fourier transform is given by $\hat{h}_{\epsilon,x}(t)=\epsilon \cos(xt)\hat{h}(\epsilon t)$.
Now fix $x=r_{k_0}(m)\in\bbR$ for some $k_0$ and assume that $\epsilon$ is sufficiently small such that there are no other eigenvalues in $[x-\sqrt{\epsilon},x+\sqrt{\epsilon}]$. Use (\ref{e:spectral2}) with the test function $h_{\epsilon,x}$.
From the previous part, the contribution of all the exceptional eigenvalues cancels out.
Since we have $||\hat{h}_{\epsilon,x}(t)||_\infty\leq\epsilon ||\hat{h}||_\infty$ we can bound the contribution of the elliptic conjugacy classes by $O(\epsilon)$. From the fast decay of $h(r)$ as $r\to\infty$ we get that the contribution of $\{r_k(m)\in\bbR|r_k(m)\neq x\}$ satisfies
\[\sum_{|r_k(m)-x|>\sqrt{\epsilon}}h_{x,\epsilon}(r_k(m))=O(\epsilon).\]
We thus get that
\[\#\{r_k(m)=x\}=\#\{r_k(\sigma m)=x\}+O(\epsilon).\]
Taking $\epsilon\to 0$ we get equality.

For the other direction, assume that we have a spectral correspondence between $V_m$ and $V_{\sigma m}$ for every $m$ for some fixed sign $\sigma\in\{\pm1\}^n$. Now assume that there is a pair $(l,\theta)$ such that $$0< \#\{\{\gamma_p\}|(l_{\gamma_p},\theta_{\gamma_p})=(l,\theta)\}\neq \#\{\{\gamma_p\}|(l_{\gamma_p},\theta_{\gamma_p})=(l,\sigma\theta)\}.$$
Also assume that $l$ is the smallest number for which this holds (recall that the set of lengthes of geodesics is discrete and that the angle is determined up to a sign by the length). Let $\hat{h}$ be a function supported on $[-l-\delta,l+\delta]$ where $\delta<l$ is sufficiently small so that there are no primitive conjugacy classes with $l<l_{\gamma_p}\leq l+\delta$, and let $H\in C^\infty((\bbR/\pi\bbZ)^n)$ be supported in a sufficiently small neighborhood of $\theta$ such that other then $(l_\gamma,\theta_\gamma)=(l,\theta)$, there are no other holonomy angles in its support with $l_{\gamma}\leq l+\delta$ (and also no angles corresponding to elliptic conjugacy classes). Since we can decompose $H$ as a sum over the functions $H_m$, using the trace formula and the same argument as before we get that
\[{\sum_{\{\gamma_p\}}}'l_{\gamma_p}\sum_{q=1}^\infty \frac{\hat h(q l_{\gamma_p})}{2\sinh(ql_{\gamma_p}/2)}H(q\theta_{\gamma_p})={\sum_{\{\gamma_p\}}}'l_{\gamma_p}\sum_{q=1}^\infty \frac{\hat h(q l_{\gamma_p})}{2\sinh(ql_{\gamma_p}/2)}H(q\sigma\theta_{\gamma_p}).\]
For all primitive conjugacy classes with $l_{\gamma_p}<l$ by our assumption of minimality we have exact cancelation, and for primitive conjugacy classes with $l_{\gamma_p}>l$ there is no contribution to either side (because $\hat h(ql_{\gamma_p})=0$). We thus get that the only contribution is from conjugacy classes with $l_{\gamma_p}=l$ and $\theta_{\gamma_p}=\theta$ on the left hand side and $\theta_{\gamma_p}=\sigma\theta$ on the right. Consequently,
$$\#\{\{\gamma_p\}|(l_{\gamma_p},\theta_{\gamma_p})=(l,\theta)\}= \#\{\{\gamma_p\}|(l_{\gamma_p},\theta_{\gamma_p})=(l,\sigma\theta)\},$$
in contradiction.
\end{proof}

Except for the case $\sigma=-1$, where complex conjugation gives the correspondence between $V_m$ and $V_{-m}$, it is not clear how to obtain such a correspondence. It is thus interesting to find for which signs $\sigma$ such a correspondence exists. We now give a condition on the lattice $\Gamma$ that implies the invariance of holonomy angles under sign changes, and hence also the spectral correspondence.

Composing the sign function $\sgn(x)=\frac{x}{|x|}$ with the determinant gives a function from $\GL(2,\bbR)$ to $\{\pm1\}$.
Note that replacing $\tau$ by $x\tau$ for $x\in \bbR^*$ does not change the sign of the determinant so that the map $\tau\mapsto\sgn(\det\tau)$ is well defined on $\PGL(2,\bbR)$.
\begin{lem}
Let $g\in \PSL(2,\bbR)$ with $|\Tr(g)|<2$ so that $g$ is conjugated in $\PSL(2,\bbR)$ to some $k\in \PSO(2)$.
Then, for any $\tau\in \GL(2,\bbR)$ we have that $\tau^{-1}g\tau\in\PSL(2,\bbR)$ is conjugated in $\PSL(2,\bbR)$ to $k^{\sgn(\det\tau)}$.
\end{lem}
\begin{proof}
This is obvious after noting that $k$ and $k^{-1}$ are conjugated in $\GL(2,\bbR)$ but not in $\SL(2,\bbR)$.
\end{proof}

Let $\tilde{G}=\GL(2,\bbR)^{n+1}$. Let $\sgn_0:\bbR^{n+1}\to\{\pm1\}^{n}$ denote the projection of the sign function to the last $n$ coordinates.
Denote by $N_{\tilde{G}}(\Gamma)$ the normalizer of $\Gamma$ in $\tilde{G}$, that is,
\[N_{\tilde{G}}(\Gamma)=\{\tau\in\GL(2,\bbR)^{n+1}|\tau^{-1}\Gamma \tau=\Gamma\}.\]
\begin{prop}\label{p:condition}
For any $\sigma\in \sgn_0(\det N_{\tilde{G}}(\Gamma))$, the holonomy angles are invariant under the sign change $\sigma$.
\end{prop}
\begin{proof}
Fix $\tau\in N_{\tilde{G}}(\Gamma)$ with $\sgn_0(\tau)=\sigma$. Fix an angle $\theta\in [-\pi,\pi]^n$ and let $\{\gamma^{(1)}\},\ldots,\{\gamma^{(k)}\}$ denote all the primitive $\Gamma$-conjugacy classes such that each $\gamma^{(i)}$ is conjugated in $\PSL(2,\bbR)^{n+1}$ to $(a_l,k_\theta)$. For each $1\leq i\leq k$ let $\tilde\gamma^{(i)}=\tau^{-1}\gamma^{(i)}\tau$. Then, by the above lemma, each $\tilde{\gamma}^{(i)}$ is conjugated in $\PSL(2,\bbR)^{n+1}$ to $(a_l,k_{\sigma\theta})$.
Since $\tau\in N_{\tilde{G}}(\Gamma)$, then $\tilde{\gamma}^{(i)}\in \Gamma$ is also a primitive element. Finally we claim that $\{\tilde\gamma^{(1)}\},\ldots,\{\tilde\gamma^{(k)}\}$ are $k$ distinct $\Gamma$-conjugacy classes. Indeed, if $\tilde\gamma^{(i)}$ and $\tilde\gamma^{(j)}$ are conjugated by some $\gamma\in \Gamma$, then $\gamma^{(i)}$ and $\gamma^{(j)}$ are conjugated by $\tau\gamma\tau^{-1}\in\Gamma$ in contradiction to the assumption that $\{\gamma^{(1)}\},\ldots,\{\gamma^{(k)}\}$ are distinct conjugacy classes.
\end{proof}

In the next section we consider lattices derived from quaternion algebras over number fields. For these lattices we can give many examples for which $\sgn_0(\det N_{\tilde{G}}(\Gamma))$ already contains all sign changes.

%*****************************************************************************************************************************************
% ****

\section{Lattices derived from quaternion algebras}\label{s:quaternion}
Let $K$ be a totally real number field. Denote by $P$ the set of places of $K$ and by $P_\infty$ and $P_f$ the set of archimedean places and finite places of $K$ respectively. For each place $\nu\in P$ let $K_\nu$ be the completion of $K$ with respect to $\nu$ and fix an embedding $\iota_\nu:K \hookrightarrow K_\nu$. For any finite set of places $S$ let $\iota_S:K\hookrightarrow \prod_{\nu\in S}K_\nu$ denote the diagonal embedding. We denote by $\calO_K$ the ring of integers of $K$ and by $\calO_K^*$ the group units.

Let $\calA$ be a quaternion algebra defined over $K$.
We denote by $\Tr_\calA$ and $n_\calA$ the relative norm and trace maps from $\calA$ to $K$, that is, $\Tr_\calA(\alpha)=\alpha+\bar\alpha,\;\;n_\calA(\alpha)=\alpha\bar\alpha$ with $\bar\alpha$ the conjugate of $\alpha$ in $\calA$. For any place $\nu$ let $\calA_\nu=\calA\otimes_{\iota_{\nu}(K)} K_\nu$. We also denote by $\iota_\nu:\calA\hookrightarrow\calA_\nu$ and $\iota_S:\calA\hookrightarrow \prod_{\nu\in S}\calA_\nu$ the corresponding embeddings. We say that $\calA$ is ramified at a place $\nu$ if $\calA_\nu$ is a division algebra and unramified if it is a matrix algebra. Denote by $\Ram(\calA),\Ram_f(\calA)$ and $\Ram_\infty(\calA)$ the set of ramified places, finite ramified places and infinite ramified places respectively. By the classification theorem of quaternion algebras, the set $\Ram(\calA)$ is always finite and even, and conversely, for any even finite set of places there is a unique (up to isomorphism) quaternion algebra ramified at these places (cf. \cite[Chapter 3 Theorem 3.1]{Vigneras80}). If the algebra $\calA$ is fixed we will just use the notation $\Ram,\Ram_f,\Ram_\infty$.

An order $\calR\subset\calA$ is a subring satisfying that $\calR\otimes K=\calA$ and that $n_\calA(\calR),\Tr_\calA(\calR)\subset \calO_K$ are in the ring of integers of $K$. We say $\calR$ is a maximal order if it is maximal with respect to inclusion; it is called an Eichler order if it is the intersection of two maximal orders. We denote by $\calR^*$ the group of invertible elements in the order $\calR$ and by $\calR^1\subset\calR^*$ the group of norm one elements inside this order.
Let $S_\infty=P_\infty\setminus \Ram_\infty$,
we then have that
$$\iota_{P_\infty}(\ker(n_\calA)\subset\prod_{\nu\in \Ram_\infty} SU(2)\times \prod_{\nu\in S_\infty}\SL(2,\bbR).$$
Moreover, the projection of $\iota_{S_\infty}(\calR^1)$ to $\prod_{\nu\in S_\infty}\PSL(2,\bbR)$ is an irreducible lattice; this lattice is called the lattice derived from $\calR$.

\subsection{Proof of Theorem \ref{t:sign2}}
Let $\calR$ be an order in a quaternion algebra. For any $\alpha\in \calR^1$ and $\beta\in \calR^*$ we have that $\beta^{-1}\alpha\beta\in \calR^1$, hence, $\calR^*\subset N_{\calA^*}(\calR^1)$. Consequently,
if $\Gamma\subset G=\prod_{\nu\in S_\infty}\PSL(2,\bbR)$ is the lattice derived from $\calR$ and $\tilde{G}=\prod_{\nu\in S_\infty}\GL(2,\bbR)$,
then $\iota_{S_\infty}(\calR^*)\subset N_{\tilde G}(\Gamma)$.
It is easy to see that this is also true for any principal congruence group, that is, for
\[\Gamma(\g)=\{\gamma\in \Gamma|\gamma\equiv I\pmod \g\},\]
with $\g\subset \calO_K$ an ideal we have that $\iota_{S_\infty}(\calR^*)\subseteq N_{\tilde{G}}(\Gamma(\g))$.

Recall that the condition given in Proposition \ref{p:condition} for the invariance of the holonomy angles under sign change depends on the signs of the determinants of the elements in $N_{\tilde G}(\Gamma)$. If $\tau=\iota_{S_\infty}(\alpha)\in N_{\tilde G}(\Gamma)$ comes from the image of some $\alpha\in \calR^*$, then $\det(\tau)=\left(\iota_\nu(n_\calA(\alpha))\right)_{\nu\in S_\infty}$. Consequently, in order to prove Theorem \ref{t:sign2} it is sufficient to show that the elements $\{\iota_{S_\infty}(n_\calA(\alpha))|\alpha\in \calR^*\}\subset \bbR^{S_\infty}$ can get all possible signs.

The following proposition characterizes the elements of $\calO_K^*$ that are obtained as norms of elements in an Eichler order.
\begin{prop}\label{p:norm}
For a quaternion algebra $\calA$ defined over $K$ denote by
$K_\calA=\{x\in K|\iota_\nu(x)>0,\;\forall\nu\in\Ram_\infty(\calA)\}$. We then have
$n_\calA(\calA)=K_\calA$ and for any Eichler order $\calR\subseteq\calA$ we have $n_\calA(\calR)=\calO_K\cap K_\calA$ and $n_\calA(\calR^*)=\calO_K^*\cap K_\calA$.
\end{prop}
\begin{proof}
The first two assertions are proved in \cite[Chapter 3 Theorem 4.1 and Corollary 5.9]{Vigneras80}. We now verify the third one. The equality $n_\calA(\calR)=\calO_K\cap K_\calA$ implies that $n_\calA(\calR^*)\subseteq \calO_K^*\cap K_\calA$. For the other direction, let $x\in \calO_K^*\cap K_\calA$ so that $x=n_\calA(\alpha)$ for some $\alpha\in\calR$.  Let $y\in\calO_K^*$ with $xy=1$. Then $y\in K_\calA$ as well and hence $y=n_\calA(\beta)$ for some $\beta\in\calR$. Now $\alpha^{-1}=\bar\alpha n_\calA(\beta)\in\calR$ so that indeed $\alpha\in\calR^*$.
\end{proof}

Let $\sgn:K^*\to\{\pm1\}^{P_\infty}$ denote the sign function, that is, $\sgn(t)=(\frac{\iota_\nu(t)}{|\iota_\nu(t)|})_{\nu\in P_\infty}$.
Let $K^+=\ker(\sgn)$ denote the group of totally positive elements and let $\calO_K^+=\calO_K^*\cap K^+$ denote the group of totally positive units.
Let $\calI(K)$ denote the group of fractional ideals in $K$, $\calP(K)$ the group of principal ideals, and $H(K)=\calI(K)/\calP(K)$ the class group of $K$. Let $\calP^+(K)$ denote the group of principal ideals generated by totally positive elements and denote by $H^+(K)=\calI(K)/\calP^+(K)$ the narrow class group. Denote by $h_K=\#H(K)$ and $h_K^+=\#H^+(K)$ the class number and narrow class number respectively.
\begin{prop}
$\sgn(\calO_K^*)=\{\pm1\}^{P_\infty}$ if and only if $h_K=h_K^+$.
\end{prop}
\begin{proof}
Let $\g_1,\g_2,\ldots,\g_{h_K}$ be a set of representatives for all ideal classes; we may assume that $\g_1=(1)$. We then have that $h_K=h_K^+$ if and only if these ideals represent all classes in the narrow class group as well.

Assume that $h_K=h_K^{+}$. We fix $\sigma\in\{\pm 1\}^{P_\infty}$ and show that there is $u\in\calO_K^*$ with $\sgn(u)=\sigma$. Let $x\in K^*$ with $\sgn(x)=\sigma$ (such an element exists from the weak approximation theorem). From our assumption, the ideal $x\calO_K$ is equivalent, in the narrow sense, to $\g_1=(1)$. This implies that there is $y\in K^+$ such that $u=xy\in \calO_K^*$, and hence $\sigma=\sgn(u)\in\sgn(\calO_K^*)$.

For the other direction we assume that $\sgn(\calO_K^*)=\{\pm 1\}^{P_\infty}$ and show that $\g_1,\ldots,\g_{h_K}$ represent all narrow classes. Fix an ideal $\g$, then
there is some $x\in K^*$ such that $x\g=\g_j$ for some $1\leq j\leq h_K$. Let $\sigma=\sgn(x)$ and fix $u\in \calO_K^*$ with $\sgn(u)=\sigma$ (which exist from our assumption). We thus have that $y=xu\in K^+$ and $y\g=xu\g=x\g=\g_i$ so $\g$ is equivalent to $\g_i$ in the narrow sense.

\end{proof}
\begin{rem}
Note that the group $(\calO_K^*)^2$ of squares of units is always contained in the positive unites $\calO_K^+$. On the other hand, Dirichlet's Unit Theorem implies that $[\calO_K^*:(\calO_K^*)^2]=2^n$. Hence the condition $h_K=h_K^{+}$ is also equivalent to $\calO_K^+=(O_K^*)^2$.
\end{rem}

The following proposition together with Proposition \ref{p:condition} complete the proof of Theorem \ref{t:sign2}.
\begin{prop}
Let $\Gamma\subset \PSL(2,\bbR)^{n+1}$ be a principal congruence group inside a lattice derived from an Eichler order $\calR$ in a quaternion algebra $\calA$ defined over a number field $K$. If $h_K=h_K^+$, then  $\sgn_0(\det(N_{\tilde{G}}(\Gamma))=\{\pm 1\}^n$.
\end{prop}
\begin{proof}
Let $d=[K:\bbQ]$ and enumerate the infinite places $P_\infty=\{\nu_0,\nu_1,\ldots,\nu_n,\nu_{n+1},\ldots,\nu_d\}$ such that $\Ram_\infty(\calA)=\{\nu_{n+1},\ldots,\nu_d\}$.
Fix $\sigma\in\{\pm 1\}^n$ and let $x\in\calO_K^*$ such that $\sgn(x)_{\nu_i}=\sigma_i$ for $i=1,\ldots ,n$ and that $\sgn(x)_{\nu_i}=1$ for $i>n$ (which exists from the assumption on the class numbers). Then $x\in \calO_{K}^*\cap K_\calA=n_\calA(\calR^*)$ so there is $\alpha\in\calR^*$ with $n_\calA(\alpha)=x$. Then  $\iota_{S_\infty}(\alpha)\in \GL(2,\bbR)^{n+1}$ lies in $N_{\tilde{G}}(\Gamma)$ and $\sgn_0\det(n_\calA(\alpha))=\sigma$.
\end{proof}
\begin{rem}
Notice that if we fix the quaternion algebra $\calA$ we can weaken the condition on the field. That is, we only require that the projection of $\sgn(\calO_K^*)$ to $\prod_{P_\infty\setminus \Ram_\infty(\calA)}\{\pm 1\}$ is onto.
\end{rem}

\subsection{Optimal Embeddings}
Before we proceed with the proof of Theorem \ref{t:Correspondence} we collect some results on optimal embeddings of quadratic orders into orders in a quaternion algebra.
Let $\calA$ be a quaternion algebra defined over a number field $K$ and $\calR\subset\calA$ a maximal order (similar results also hold for Eichler orders, but for simplicity we will restrict the discussion to maximal orders). We further assume that $\calA$ satisfies the Eichler condition, that is, $\calA$ is unramified in at least one infinite place.

Let $L/K$ be a quadratic extension, we say that $L$ embeds into $\calA$ if there is a nontrivial $K$-homomorphism from $L$ to $\calA$. We note that this happens if and only if $L_\nu$ is a quadratic field extension of $K_\nu$ for any $\nu\in \Ram(\calA)$ (cf. \cite[Chapter III Theorem 4.1]{Vigneras80}). Let $\calO_L$ denote the ring of algebraic integers in $L$ and let $\calO\subset\calO_L$ be an order (i.e., a subring satisfying that $\calO\otimes K=L$). We say that an embedding $\phi:L \into \calA$ is an optimal embedding of $\calO$ in $\calR$ if $\phi(L)\cap \calR=\phi(\calO)$. Note that if $\phi$ is an optimal embedding, then for any $x\in N_{\calA}(\calR)$ the map $\phi_x(u):=x^{-1}\phi(u) x$ is also an optimal embedding.
\begin{defn}
Denote by $m^*(\calO,\calR)$ the number of optimal embeddings of $\calO$ into $\calR$ modulo conjugation by elements of $\calR^*$,
and by $m^1(\calO,\calR)$ the number of optimal embeddings of $\calO$ into $\calR$ modulo conjugation by elements of $\calR^1$.
\end{defn}

For any finite place $\nu\in P_f$ let $\calO_\nu=\calO\otimes_{\calO_K} \calO_{K_\nu}$, let $\calR_\nu=\calR\otimes_{\calO_K} \calO_{K_\nu}$, and let $m^*(\calO_\nu,\calR_\nu)$ denote the number of optimal embeddings of $\calO_\nu$ in $\calR_\nu$ modulo $\calR_\nu^*$. For places where $\calA$ is unramified, $\calA_\nu=\Mat(2,K_\nu)$ and $m^*(\calO_\nu,\calR_\nu)=1$ \cite[Chapter II Theorem 3.2]{Vigneras80}. For $\nu\in\Ram(\calA)$, $\calA_\nu$ is a division algebra and by \cite[Chapter II Theorem 3.1]{Vigneras80} $m^*(\calO_\nu,\calR_\nu)=0$ unless $\calO_\nu=\calO_{L_\nu}$ in which case
\begin{equation}\label{e:optimallocal}m^*(\calO_{L_\nu},\calR_\nu)=\left\lbrace\begin{array}{cc} 1 & \nu \mbox{ ramifies in } L \\ 2 & \nu \mbox{ is unramified in } L\end{array}\right..
\end{equation}

The global number of optimal embeddings can be obtained from this local data as follows:
\begin{prop}\label{p:optimal}
Let $\calA$ be a quaternion algebra over $K$ and $\calR\subset\calA$ a maximal order.
Let $L/K$ be a quadratic extension embedded in $\calA$ and $\calO\subset\calO_L$ an order.
If $\calA=\Mat(2,K)$ we further assume that there is an infinite place with $L_\nu=\bbC$. Then
\[m^1(\calO,\calR)=\frac{h(\calO)}{h_K}[\calO_K^*:n_{L/K}(\calO^*)]2^{-r_\calA}\prod_{\nu\in \Ram_f(\calA)}m^*(\calO_\nu,\calR_\nu),\]
where $r_\calA=\#\Ram_\infty(\calA)$, $h_K$ is the class number of $K$ and $h(\calO)$ is the class number of $\calO$.
\end{prop}
\begin{proof}
Let $H(K)$ denote the class group of $K$ and $H(K_\calA)$ the narrow class group corresponding to $K_\calA$
(that is the quotient of the fractional ideals by the principal ideals generated by elements of $K_\calA$). Let $h_K$ and $h_{K_\calA}$ denote the corresponding class number and narrow class number.
By \cite[Chapter III Proposition 5.16]{Vigneras80}  $m^*(\calO,\calR)$ does not depend on the choice of maximal order $\calR$, and by \cite[Chapter III Theorem 5.15]{Vigneras80} it is given by
\begin{equation}\label{e:optimal1}
m^*(\calO,\calR)=\frac{h(\calO)}{h_{K_\calA}}\prod_{\nu\in \Ram_f(\calA)}m^*(\calO_\nu,\calR_\nu).
\end{equation}
By \cite[Chapter III Corollary 5.13]{Vigneras80} we have
\begin{equation}\label{e:optimal2}
m^1(\calO,\calR)=m^*(\calO,\calR)[n_\calA(\calR^*):n_{L/K}(\calO^*)].
\end{equation}
Combining (\ref{e:optimal1}) and $(\ref{e:optimal2})$ we get
\begin{equation}\label{e:optimal3}
m^1(\calO,\calR)=\frac{[\calO_K^*:n_{L/K}(\calO^*)]}{[\calO_K^*:n_\calA(\calR^*)]}\frac{h(\calO)}{h_{K_\calA}}\prod_{\nu\in \Ram_f(\calA)}m^*(\calO_\nu,\calR_\nu).
\end{equation}

We can identify the principal fractional ideals in $K^*$ with $K^*/\calO_K^*$.
Hence, the kernel of the natural projection $\pi:H(K_\calA)\to H(K)$ satisfies
$$\ker(\pi)\cong (K^*/\calO_K^*)/(K_\calA/\calO_K^*\cap K_\calA)=(K^*/\calO_K^*)/(K_\calA/n_\calA(\calR^*),$$
which is a finite group of order
$$\#\ker(\pi)=\frac{[K^*:K_\calA^*]}{[\calO_K^*:n_\calA(\calR^*)]}=\frac{2^{r_\calA}}{[\calO_K^*:n_\calA(\calR^*)]}.$$
Now plug in
$h_{K_\calA}=\#\ker(\pi) h_{K}=2^{r_\calA}\frac{h_K}{[\calO_K^*:n_\calA(\calR^*)]}$
in (\ref{e:optimal3}) to get  the result.
\end{proof}

The number of conjugacy classes in $\Gamma$ with a given trace can be expressed as a sum over the number of optimal embeddings of certain quadratic orders.
\begin{defn}\label{d:ODd}
For any quadratic extension $L=K(\sqrt{D})$ of $K$ (with $D\in K$ not a square) and any ideal $d$ in $\calO_K$ let $\calO_{D,d}\subset\calO_L$ be the order with relative discriminant $d$, that is
\[\calO_{D,d}=\set{\frac{t+u\sqrt{D}}{2}\in\calO_K,\;\big{|}\;d|(u^2D)}.\]
(Note that this definition does not depend on the choice of $D$ modulo squares in $K^*$.)
\end{defn}
\begin{lem}\label{l:count1}
Let $\Gamma$ be a lattice derived from an order $\calR$ in a quaternion algebra $\calA$ over $K$ and let $S=P_\infty\setminus\Ram_\infty(\calA)$.
Let $t\in\calO_K$ and let $D=t^2-4$. Assume that $D\neq 0$, then
\[\sharp\{\{\gamma\}\in\Gamma^\#:\Tr(\gamma)=\pm\iota_{S}(t)\}=\sum_{d||(D)}m^1(\calO_{D,d},\calR),\]
where the notation $d||(D)$ means that $(D)=d\f^2$ for $\f\subseteq\calO_K$ an ideal.

\end{lem}
\begin{proof}
Any $\gamma\in\Gamma$ comes from the projection of $\iota_S(\alpha)$ for some $\alpha\in\calR^1$. Note that both $\alpha$ and $-\alpha$ give the same element in $\Gamma$ so we can count the number of conjugacy classes in $\calR^1$ with trace equal to $t$. For $\alpha\in \calR^1$ let $t_\alpha=\Tr_\calA(\alpha)\in K$ and $D_\alpha=t_\alpha^2-4$. When $\alpha$ is not in the center, $L=K(\sqrt D_\alpha)$ is a quadratic extension that embeds in $\calA$ via $\phi_\alpha(\sqrt{D_\alpha})=2\alpha-t_\alpha$. Let $\g_\alpha=\{u\in L|\exists x\in K,\;x+u\alpha\in\calR\}$ and $d_\alpha=\g_\alpha^2D_\alpha$. Then $d_\alpha,\g_\alpha^{-1}\subset\calO_K$ are integral ideals (so  $d_\alpha|| (D_\alpha)$) and $\phi_\alpha(\calO_{D_\alpha,d_\alpha})=\calR^1\cap \phi(L)$ is an optimal embedding (cf. \cite[Lemma 4.2 and Proposition 4.3]{KelmerSarnak09}). Now
\begin{eqnarray*}
\sharp\{ \{\alpha\}|\Tr(\alpha)=t\}=\sum_{d||(D)}\sharp\{\{\alpha\}|\Tr(\alpha)=t,\;d_\alpha=d\}
\end{eqnarray*}
We thus need to show that $$\sharp\{\{\alpha\}|\Tr(\alpha)=t,\;d_\alpha=d\}=m^1(\calO_{D,d},\calR).$$
For any $\alpha$ with $t_\alpha=t$ and $d_\alpha=d$ we get an optimal embedding $\phi_\alpha$ as above (and conjugating $\alpha$ gives a conjugate embedding). On the other hand, given an optimal embedding $\phi:\calO_{D,d}\to\calA$ we have that $\alpha=\phi(\frac{t+\sqrt{D}}{2})$ is the unique element in the image with trace $t_\alpha=t$ and $d_\alpha=d$.
\end{proof}

\subsection{Proof of Theorem \ref{t:Correspondence}}\label{s:geometric}
Let $\Gamma\subset\PSL(2,\bbR)^{n+1}$ be a lattice derived form a maximal order $\calR$ in a quaternion algebra $\calA$ over a number field $K$. This means that $\calA$ is unramified in $n+1$ infinite places, that we denote by $S=\{\nu_0,\nu_1,\ldots,\nu_n\}$. Further assume that $n$ is even. We can thus find another quaternion algebra $\tilde\calA$ such that $\Ram_f(\tilde \calA)=\Ram_f( \calA)$ and that $\Ram_\infty(\tilde\calA)=\calP_\infty\setminus\{\nu_0\}$. Let $\tilde{S}=\{\nu_1,\ldots, \nu_n\}$.
Let $\tilde\Gamma\subset\PSL(2,\bbR)$ be a lattice derived from a maximal order $\tilde\calR\subset\tilde \calA$.

Recall the correspondence between closed geodesics and conjugacy classes. Each closed geodesic $C$ in $\Sigma_0$ corresponds to a conjugacy class of some hyperbolic-elliptic $\gamma\in\Gamma$. Any such $\gamma$ is the projection to $\PSL(2,\bbR)^{n+1}$ of $\iota_{S}(\alpha)$ for some $\alpha\in\calR^1$ with trace $t=\Tr_\calA(\alpha)\in\calO_K$ satisfying that $|\iota_{\nu_j}(t)|<2$ for $\nu_j\in \tilde{S}$. The embeddings $\iota_{\nu_j}(t)\in\bbR$ correspond to the holonomy angles $\theta_{C,j}$ by $|\iota_{\nu_j}(t)|=|2\cos(\theta_{C,j}/2)|,$ for $j=1,\ldots,n$. The additional embedding $\iota_{\nu_0}(t)$ satisfies $|\iota_{\nu_0}(t)|=2\cosh(l_C/2)$.

On the other hand, a closed geodesic in $T^1\tilde \calM$ corresponds to a conjugacy class of some hyperbolic  $\gamma\in\tilde \Gamma$. Any such $\gamma$ is the projection to $\PSL(2,\bbR)$ of $\iota_{\nu_0}(\alpha)\in\SL(2,\bbR)$ for some $\alpha\in\tilde\calR^1$ and the length of the corresponding geodesic satisfies $|\iota_{\nu_0}(\Tr_\calA(\alpha))|=2\cosh(l_{\gamma}/2)$. Let $u_\gamma\in\PSU(2)^n$ be the projection of $\iota_{\tilde{S}}(\alpha)\in\prod_{\nu\in \tilde{S}}\calR_\nu^1\subset \SU(2)^n$.
For any closed geodesic $\{\gamma\}$ in $\tilde{\calM}=\tilde\G\bs \bbH$ we attache the conjugacy class of $u_\gamma$ in $\PSU(2)^n$. Theorem \ref{t:Correspondence} is thus equivalent to the following proposition:
\begin{prop}\label{p:correspondence1}
Let $t\in \calO_K$. Then  $\iota_{\nu_0}(t)\in \Tr(\tilde\Gamma)$ if and only if $\iota_{S_\infty}(t)\in \Tr(\Gamma)$ and $|\iota_\nu(t)|<2$ for all $\nu\in \tilde{S}$. Moreover, for any $t\in\calO_K$ such that $\iota_{\nu_0}(t)\in\Tr(\tilde\Gamma)$
$$\sharp\{\{\gamma\}\in \Gamma^\#|\Tr(\gamma)=\pm \iota_{S_\infty}(t)\}=2^n\sharp\{\{\gamma\}\in \tilde \Gamma^\#|\Tr(\gamma)=\pm\iota_{\nu_0}(t)\}.$$
\end{prop}
\begin{proof}
Let $t\in\calO_K$.
Assume that $\iota_S(t)\in\Tr(\Gamma)$ and that $|\iota_\nu(t)|<2$, $\forall\nu\in \tilde{S}$.
Then there is $\alpha\in \calR^1$ such that $t=\Tr_\calA(\alpha)$.
The quadratic extension $L=K(\alpha)$ naturally embeds in $\calA$. Since we assume $|\iota_\nu(t)|<2$ for $\nu\in \tilde S$, then $L_\nu=\bbC$ for $\nu\in \tilde S$ and hence $L$ also embeds in $\tilde\calA$. Let $\calO=L\cap\calR\subset \calO_L$. By definition, $\calO$ is optimally embedded in $\calR$ and hence $m^1(\calO,\calR)>0$. By Proposition \ref{p:optimal} we get that $2^nm^1(\calO,\tilde\calR)=m^1(\calO,\calR)$ so $m^1(\calO,\tilde\calR)> 0$ as well and hence there is an optimal embedding  $\phi:\calO\to \tilde\calR$. We thus get that $t=\Tr_{\tilde\calA}(\phi(\alpha))\in \Tr_{\tilde\calA}(\tilde\calR^1)$ and hence $\iota_{\nu_0}(t)\in\Tr(\tilde\Gamma)$.

For the other direction, the same argument shows that $\Tr(\tilde\calR^1)\subset\Tr(\calR^1)$ and since $\tilde\calA$ is ramified at any place $\nu\in \tilde S$, then any $t\in \Tr_{\tilde\calA}(\tilde R^1)$ satisfies $|\iota_\nu(t)|<2$ for $\nu\in \tilde S$.

Next, from Lemma \ref{l:count1} we can express the number of $\Gamma$-conjugacy classes with $\Tr(\gamma)=\pm\iota_{S}(t)$ as a sum over ideals $d|(D)$ of the quantities $m^1(\calO_{D,d},\calR)$, where $D=t^2-4$.
By Proposition \ref{p:optimal}, for each of the orders $\calO_{D,d}$ appearing in this sum we have
\[m^1(\calO_{D,d},\calR)=2^nm^1(\calO_{D,d},\tilde\calR),\]
implying that
\[\sharp\{\{\gamma\}\in\Gamma^\#|\Tr(\gamma)=\pm\iota_{S}(t)\}=2^n\sharp\{\{\gamma\}\in\tilde\Gamma^\#|\Tr(\gamma)=\pm\iota_{\nu_0}(t)\}.\]
\end{proof}

\begin{rem}\label{r:dist}
Taking the above correspondence into account, Theorem \ref{t:equi} implies that
\[\frac{1}{\pi_p(x)}\mathop{\sum_{\{\gamma\}\in{\tilde\Gamma}^\#}}_{l_\gamma\leq x}f(u_\gamma)=\int f(u)du+O(\pi_p(x)^{-1/4}),\]
where the integral is with respect to the Haar measure on $\PSU(2)^n$, $f\in C^\infty(\PSU(2)^n)$ is invariant under conjugation and $\pi_p(x)\sim \Li(e^x)$ is the number of primitive closed geodesic on $\tilde \calM$ with length bounded by $x$.
We note that such an equidistribution result holds in a much more general setting.
Let $\tilde\Gamma\in \PSL(2,\bbR)$ be any lattice and let $\rho:\Gamma\to U$ denote a homomorphism with dense image in some compact group $U$. To each closed geodesic $\{\gamma\}$ on $\tilde\Gamma\bs\bbH$ attach the conjugacy class of $\rho(\gamma)$ in $U$. As the length of the geodesic grows the corresponding conjugacy classes become equidistributed with respect to Haar measure. However, giving a rate of equidistribution in this case requires a spectral gap for irreducible lattices in $\SL(2,\bbR)\times U$ which is not known in this generality.
\end{rem}

\section{Some applications}\label{s:applications}
\subsection{A Spectral Correspondence}\label{s:spectral}
The first application is a spectral correspondence (that is a special case of the Jacquet-Langlands correspondence) between $V_m(\Gamma\bs G)$ and  $L^2(\tilde\Gamma\bs\bbH,\rho_m)$.

Retain the notation of section \ref{s:geometric}.
Let $m\in\bbN^n$ and for each $j=1,\ldots,n$ let $\rho_{m_j}$ denote the irreducible representation of $\PSU(2)$ of dimension $2m-1$. Let $$\rho_m=\rho_{m_1}\otimes\cdots\otimes\rho_{m_n},$$
denote the corresponding $|m|^*$-dimensional representation of $\PSU(2)^n$. Composing this representation with the homomorphism $\gamma\mapsto u_\gamma$ gives a representation of $\Gamma$, that we still denote by $\rho_m$.
The character of this representation is given by
$$\chi_{\rho_m}(\gamma)=\Tr(\rho_m(\gamma))=\prod_{j=1}^n\frac{\sin((m_j-\tfrac{1}{2})\theta_{u_{\gamma,j}})}{\sin(\theta_{u_{\gamma,j}}/2)}=F_m(\theta_{u_\gamma}).$$
Let $L^2(\tilde\G\bs \bbH,\rho_m)$ denote the space of functions $\psi:\bbH\to \bbC^{|m|^*}$ that are square integrable on $\calF_{\tilde\Gamma}$ and transform under $\tilde\Gamma$ via
\[\psi(\gamma z)=\rho_m(\gamma)\psi(z).\]
For any $\sigma\in\{\pm\}^n$ let $V_{\sigma m}(\G\bs G)\subset L^2(\G\bs G,\sigma m)$ be as defined in (\ref{e:Vm}).
We then have the following spectral correspondence:
\begin{cor}
There is a $2^n$ to one correspondence between the spectrum of $\bigoplus_{\sigma} V_{\sigma m}(\G\bs G)$ and the spectrum of   $L^2(\tilde\G\bs \bbH,\rho_m)$.
\end{cor}
\begin{proof}
Let $\{\psi_k\}_{k=0}^\infty$ be a basis for $L^2(\tilde\G\bs \bbH,\rho_m)$ composed of eigenfunctions $\lap \psi_k+\lambda_k\psi_k=0$. With the parametrization $\lambda_k=\tfrac{1}{4}+\tau_k^2$, the trace formula in this setting takes the form
\begin{eqnarray}\label{e:TF}
\quad \sum_{k} h(\tau_k)&=&\frac{|m|^*\vol(\calF_{\tilde\Gamma})}{(4\pi)}\int_\bbR h(r)r\tanh(\pi r)dr\\
\nonumber && + \mathop{\sum_{\{\gamma\}}}_{\mathrm{hyperbolic}} \frac{l_{\gamma_p} \hat{h}(l_\gamma)F_m(\theta_{u_\gamma})}{2\sinh(l_\gamma/2)}
+\mathop{\sum_{\{\gamma\}}}_{\mathrm{elliptic}} \frac{\tilde{h}(\theta_\gamma,0)F_m(\theta_{u_\gamma})}{M_{\gamma_p}\sin(\theta_\gamma/2)}
\end{eqnarray}
Recall the correspondence between conjugacy classes in $\tilde\Gamma$ and the orders $\calO_{D,d}$ described in Lemma \ref{l:count1}. For $\gamma$ hyperbolic we have that $l_{\gamma_p}=\vol(\Gamma_\gamma\bs G_\gamma)=\Reg(\calO_{D,d})$ is given by the regulator (see e.g. \cite[Page 36]{Efrat87}). For an elliptic element $\gamma$ the group $\tilde\Gamma_\gamma\cong\calO_{D,d}^1/\{\pm I\}$ is a finite cyclic group generated by $\gamma_p$ and hence $\vol(\tilde\Gamma_\gamma\bs G_\gamma)=\frac{1}{M_{\gamma_p}}=\frac{2}{\#\calO_{D,d}^1}$.
We can rewrite the sum over the hyperbolic conjugacy classes as
\[\mathop{\sum_{\{\gamma\}}}_{\mathrm{hyperbolic}}\frac{l_{\gamma_p} \hat{h}(l_\gamma)}{2\sinh(l_\gamma/2)}
F_m(\theta_\gamma)=\mathop{\sum_{t\in\calO_K}}_{|\iota_{\nu_0}(t)|>2}g(t)\sum_{d|(t^2-4)}\reg(\calO_{D,d})m^1(\calO_{D,d},\tilde\calR),\]
and the sum over the elliptic classes as
\[\mathop{\sum_{\{\gamma\}}}_{\mathrm{elliptic}} \frac{\tilde{h}(\theta_\gamma,0)}{M_{\gamma_p}\sin(\theta_\gamma)}
F_m(\theta_{u_\gamma})=\mathop{\sum_{t\in\calO_K}}_{|\iota_{\nu_0}(t)|<2} g(t)\sum_{d|(t^2-4)}\frac{2 m^1(\calO_{D,d},\tilde\calR)}{\#\calO_{D,d}^1} ,\]
where the function $g:\calO_K\to\bbR$ is defined by
$$g(t)=\left\lbrace\begin{array}{cc}
\frac{\hat{h}(l_\gamma)F_m(\theta_\gamma)}{2\sinh(l_\gamma/2)}& \exists \gamma\in\tilde\Gamma,\;|\Tr(\gamma)|=|\iota_{\nu_0}(t)|>2\\
&\\
\frac{\tilde{h}(\theta_\gamma,0)F_m(\theta_{u_\gamma})}{\sin(\theta_\gamma/2)}
 & \exists \gamma\in\tilde\Gamma,\;|\Tr(\gamma)|=|\iota_{\nu_0}(t)|<2\\
 &\\
0 & \mbox{ otherwise.}\end{array}\right.$$
Recall that $2\cos(\theta_{u_\gamma,j})=\iota_{\nu_j}(t)$ is determined by $\iota_{\nu_0}(t)$ so this is well defined.

Now compare this to the hybrid trace formula for $V_m(\G\bs G)$ given in Theorem \ref{t:Htrace}', that is,
\begin{eqnarray}\label{e:HTF}
\lefteqn{\sum_{\sigma\in\{1,-1\}^n}\sum_{k} h(r_{k}(\sigma m))+\delta_{m,1}2^{n} h(i/2)}\\
\nonumber = &&\frac{2^n|m|^*\vol(\calF_{\Gamma})}{(4\pi)^{n+1}}\int_\bbR h(r)r\tanh(\pi r)dr
 +{\sum_{\{\gamma\}}}' \frac{l_{\gamma_p} \hat{h}(l_\gamma)}{2\sinh(l_\gamma/2)}F_m(\theta_\gamma)\\
\nonumber +&&{\sum_{\{\gamma\}}}'' \frac{\tilde{h}(\theta_{\gamma_0},0)F_m(\theta_{\gamma}/2)}{M_{\gamma_p}\sin(\theta_{\gamma_0}/2)}
\end{eqnarray}
From Proposition \ref{p:correspondence1} the sum over the hyperbolic-elliptic and the elliptic conjugacy classes can be written as
\[\mathop{\sum_{t\in\calO_K}}_{|\iota_{\nu_0}(t)|>2}g(t)\sum_{d|(t^2-4)}\reg(\calO_{D,d})m^1(\calO_{D,d},\calR),\]
and
\[\mathop{\sum_{t\in\calO_K}}_{|\iota_{\nu_0}(t)|<2} g(t)\sum_{d|(t^2-4)}\frac{2m^1(\calO_{D,d},\calR)}{\#\calO_{D,d}^1}\]
respectively, with $g(t)$ as above. From Proposition \ref{p:optimal} for all orders $\calR$ appearing in the sum we have that $m^1(\calO_{D,d},\calR)=2^n m^1(\calO_{D,d},\tilde\calR)$ implying that
\begin{eqnarray*}2^n\mathop{\sum_{\{\gamma\}\in\tilde\Gamma^\#}}_{\mathrm{hyperbolic}} \frac{l_{\gamma_p} \hat{h}(l_\gamma)}{2\sinh(l_\gamma/2)}
F_m(\theta_\gamma)+2^n\mathop{\sum_{\{\gamma\}\in\tilde\Gamma^\#}}_{\mathrm{elliptic}} \frac{\tilde{h}(\theta_\gamma,0)}{M_{\gamma_p}\sin(\theta_\gamma/2)}
F_m(\theta_{u_\gamma})\\
=\mathop{\sum_{\{\gamma\}\in\Gamma^\#}}_{\mathrm{hyperbolic-elliptic}} \frac{l_{\gamma_p} \hat{h}(l_\gamma)}{2\sinh(l_\gamma/2)}
F_m(\theta_\gamma)+\mathop{\sum_{\{\gamma\}\in\Gamma^\#}}_{\mathrm{elliptic}} \frac{\tilde{h}(\theta_{\gamma_0},0)F_m(\theta_{\gamma})}{M_{\gamma_p}\sin(\theta_{\gamma_0}/2)}.
\end{eqnarray*}
Comparing the volumes of fundamental domains (given in \cite[Chapter IV Corollary 1.8]{Vigneras80}) we get $\vol(\calF_{\tilde\Gamma})=\frac{\vol(\calF_{\Gamma})}{(4\pi)^n}$.
Consequently, after multiplying (\ref{e:TF}) by $2^n$ its right hand side equals the right hand side of (\ref{e:HTF}).
We thus get the equality
\begin{eqnarray}\label{e:spectral}
\sum_{\sigma\in\{1,-1\}^n}\sum_{k} h(r_{k}(\sigma m))+\delta_{m,1}(2)^{n} h(i/2)=2^n\sum_{k}h(\tau_k).
\end{eqnarray}
(Note that $\tau_0=\frac{i}{2}$ appears on the right hand side only when $m=(1,\ldots,1)$, in which case it also appears on the left hand side.)
Since this identity holds for all test functions this proves the correspondence.
\end{proof}
\begin{rem}
As there are $2^n$ possible sign changes, this correspondence suggests that actually the spectrum of $L^2(\tilde\G\bs G,\rho_m)$ and $V_{\sigma m}$ are the same for each $\sigma$. If we assume that $h_K=h_K^{+}$, this follows from
Theorem \ref{t:sign2}.
\end{rem}
\begin{rem}
We remark that when taking $m=(0,1,\ldots,1)$, on one side of this correspondence we just get the standard Maass forms for the group $\tilde{\Gamma}\subset\PSL(2,\bbR)$, but on the other side we do not get Maass forms but rather the hybrid forms corresponding to this weight. This is very different from the case of quaternion algebras defined over $\bbQ$ where Maass forms are lifted to Maass forms and holomorphic forms are lifted to holomorphic forms.
\end{rem}
\subsection{Results on Hilbert modular groups}
Next we prove analogous results to Theorems \ref{t:PGT0} and \ref{t:equi} for Hilbert modular groups, that is, for lattices of the form $\Gamma=\PSL(2,\calO_K)$ with $K$ a totally real number field. These lattices are not cocompact, nevertheless, we can still get these results without using the trace formula in the non-compact setting (as long as $[K:\bbQ]>2$).

Let $\Gamma=\PSL(2,\calO_K)$ with $K$ totally real with $[K:\bbQ]=n+1>2$. This lattice is the lattice derived from the maximal order $\Mat(2,\calO_K)$ in the matrix algebra $\Mat(2,K)$. Let $\nu_0,\ldots,\nu_n$ denote the infinite places of $K$ and let $\tilde\Gamma$ denote a lattice derived from a maximal order in a quaternion algebra that is ramified in two infinite places $\{\nu_{n-1},\nu_n\}$. We then have a homomorphism $\gamma\mapsto u_\gamma$ of $\tilde\Gamma$ into $\PSU(2)^2$ defined as in the proof of Theorem \ref{t:Correspondence}. For any hyperbolic-elliptic $\gamma\in\tilde\Gamma$ let $\theta_{\gamma,1},\ldots,\theta_{\gamma,n-2}$ be the corresponding holonomy angles and let $\theta_{\gamma,n-1},\theta_{\gamma,n}$ denote two additional angles corresponding to $u_\gamma\in\PSU(2)^2$ (these last two angles are only defined up to a sign). Just as in the proof of Proposition \ref{p:correspondence1} there is a correspondence between $\Gamma$-conjugacy classes and $\tilde\Gamma$-conjugacy classes which implies the following equality:
\[{\sum_{\{\gamma\}\in\Gamma^\#}}'\frac{l_{\gamma_p}h(l_\gamma)F_m(\theta_\gamma)}{2\sinh(l_\gamma/2)}= 4{\sum_{\{\gamma\}\in\tilde\Gamma^\#}}'\frac{l_{\gamma_p}h(l_\gamma)F_m(\theta_\gamma)}{2\sinh(l_\gamma/2)}\]
for any $m\in\bbN^{n}$ and any $\hat{h}$ compactly supported, where in both cases the sum is over the hyperbolic-elliptic conjugacy classes in the corresponding lattice.
Notice that the right hand side is exactly the expression appearing in the Hybrid trace formula on the space $V_{m'}(\tilde\Gamma\bs G,\rho_{m''})$ with $m'=(0,m_1,\ldots,m_{n-2})$, $m''=(m_{n-1},m_{n})$, and $\rho_{m''}$ is the representation of $\Gamma$ obtained from composing the corresponding irreducible representation of $\PSU(2)^2$ with the map $\gamma\to u_\gamma$ above.
Using this formula and following the same arguments as in the proofs of Theorems \ref{t:PGT0} and \ref{t:equi} gives the following result:
\begin{cor}\label{c:Hilbert}
Let $\Gamma=\PSL(2,\calO_K)$ with $K$ totally real with degree $[K:\bbQ]=n+1>2$. We then have for any $f\in C^\infty((\bbR/2\pi\bbZ)^n)$ that is invariant under all sign changes
\[{\sum_{l_{\gamma_p}<x}}'f(\theta_{\gamma_p})=2^n\Li(e^x)\mu(f)+O(e^{3x/4}),\]
where the sum is over the primitive hyperbolic-elliptic conjugacy classes in $\Gamma$.
\end{cor}
\begin{rem}
In fact, we need only to assume that $f$ is invariant under sign changes in two of the coordinates.
Even this weaker assumption (and also the assumption $[K:\bbQ]>2$) can probably be removed. However, for this it seems that the non-compact trace formula is essential.
\end{rem}
\subsection{Distribution of fundamental units}
Another application is a result on the distribution of fundamental units of orders in quadratic extensions of number fields. Let $K$ denote a totally real number field and let $P_\infty=\{\nu_0,\ldots,\nu_n\}$ denote the infinite places with $\iota_{\nu_j}:K\into \bbR$ the corresponding embeddings.
For any pair $(D,d)$ with $D\in K^*/(K^*)^2$ and $d\subset\calO_{K(\sqrt{D})}$ an ideal let $\calO_{D,d}\subset\calO_L$ denote the order with relative discriminant $d$ (see Definition $\ref{d:ODd}$).
For $D$ satisfying that $\iota_{\nu_0}(D)>0$ and $\iota_{\nu_j}(D)<0$ for $j\geq 1$ the group of relative norm one elements $\calO_{D,d}^1$ is isomorphic to $\bbZ\times\bbZ/2\bbZ$. We say that $\epsilon_{D,d}\in \calO_{D,d}^1$ is the fundamental unit of $\calO_{D,d}$ if $\iota_{\nu_0}(\epsilon_{D,d})>1$ and any $\alpha\in \calO_{D,d}^1$ can be written as $\alpha=\pm \epsilon_{D,d}^k$ for some $k\in\bbZ$.  For $j\geq 1$ we have that $\iota_{\nu_j}(\epsilon_{D,d})=e^{i\theta_j(D,d)}$ lies on the unit circle and we denote $\theta(D,d)=(\theta_1(D,d),\ldots,\theta_n(D,d))\in [-\pi,\pi]^n$. As a corollary of Theorem \ref{t:equi} we get that after adding appropriate weights, these angles become equidistributed with respect to the measure $\mu$.
\begin{cor}\label{c:units}
Let $\calR$ be a maximal order in a quaternion algebra $\calA$ defined over $K$ with $\Ram_\infty(\calA)=\emptyset$ (if $\Ram(\calA)=\emptyset$ further assume that $[K;\bbQ]>2$). Let $f\in C^\infty(\bbR/2\pi\bbZ)$ be invariant under all sign changes $f(\theta)=f(\sigma\theta)$.
We then have
\[{\mathop{\sum_{(D,d)}}_{\iota_0(\epsilon_{D,d})\leq T}} m^1(\calO_{D,d},\calR)f(\theta(D,d))=2^{[K:\bbQ]-2}\Li(T^2)\mu(f)+O(T^{3/2}),\]
where the sum is over pairs $(D,d)$ with $D\in K^*/(K^*)^2$ satisfying that $\iota_0(D)>0$ and $\iota_j(D)<0$, for $j\geq 1$.
\end{cor}
\begin{proof}
Let $\Gamma\subset\PSL(2,\bbR)^{n+1}$ be the lattice derived from $\calR$, where $n=[K:\bbQ]-1$.
By Theorem \ref{t:equi} (or Corollary \ref{c:Hilbert}, for $\calR$ a maximal order in a matrix algebra) we have
\[{\sum_{l_{\gamma_p}<x}}'f(\theta_\gamma)=2^n\Li(e^x)\mu(f)+O(e^{3x/4}).\]
(We recall that for these lattices there are good bounds for the spectral gap.)
Any primitive element $\gamma_p$ is the projection of $\iota_{P_\infty}(\alpha)$ for some $\alpha\in \calR^1$.
Let $D=\Tr(\alpha)^2-4$ and $d=d_\alpha\subset\calO_K$ be as in the proof of Lemma \ref{l:count1}.
Then $\alpha\in \calO_{D,d}^1$ and since we assume that $\gamma_p$ is primitive, then $\epsilon_{D,d}=\pm\alpha^{\pm1}$.
Next note that $2\cosh(l_{\gamma_p}/2)=|\iota_{\nu_0}(\Tr(\alpha))|=\iota_{\nu_0}(\epsilon_{D,d}+\epsilon_{D,d}^{-1})$ so that $\iota_{\nu_0}(\epsilon_{D,d})=e^{l_{\gamma_p}/2}$. By the same argument we have that $e^{i\theta_j(D,d)}=\iota_{\nu_{j}}(\epsilon_{D,d})=e^{\pm i\theta_{\gamma,j}}$.  Finally, for each optimal embedding of $\calO_{D,d}$ in $\calR$ we get two primitive conjugacy classes (corresponding to $\epsilon_{D,d}$ and $\epsilon_{D,d}^{-1}$). We thus get that
\[{\sum_{l_{\gamma_p}<x}}'f(\theta_\gamma)={\mathop{\sum_{(D,d)}}_{\iota_0(\epsilon_{D,d})\leq e^{x/2}}} 2m^1(\calO_{D,d},\calR)f(\theta(D,d)).\]
Setting $x=2\ln(T)$ gives the result.
\end{proof}

For the case where $\calR=\Mat(2,\calO_K)$, we can interpret the number of optimal embeddings $m^1(\calO_{D,d},\calR)$ as the number of equivalence classes of quadratic forms with discriminant $D$ and primitive discriminant $d$ (cf. \cite[Section 7]{Efrat87}). In particular, Corollary \ref{c:units} gives the asymptotic growth of the number of classes of quadratic forms (that are indefinite in one imbedding and definite in the rest) with corresponding fundamental units in a prescribed set.

%----------------------------------------------------------------
%%GATHER{C:/Bib/Mybib.bib}   % For Gather Purpose Only
%\bibliographystyle{amsplain}
%\bibliography{C:/Bib/Mybib}
%\bibliography{/zf/kelmerdu/Documents/tex/Bib/Mybib}

%*********************************************************************
\def\cprime{$'$} \def\cprime{$'$}
\providecommand{\bysame}{\leavevmode\hbox to3em{\hrulefill}\thinspace}
\providecommand{\MR}{\relax\ifhmode\unskip\space\fi MR }
\providecommand{\MRhref}[2]{%
  \href{http://www.ams.org/mathscinet-getitem?mr=#1}{#2}
}
\providecommand{\href}[2]{#2}

%*********************************************************************

\end{document}